\documentclass{amsart}
\pdfoutput=1

\usepackage{amssymb}
\usepackage{microtype}
\usepackage[numbers]{natbib}
\usepackage{enumerate}
\usepackage[all,tips]{xy}
\SelectTips{cm}{10}
\usepackage{aliascnt}

\usepackage{hyperref}
\hypersetup{
  pdftitle={Geometric variational crimes: Hilbert complexes, finite
    element exterior calculus, and problems on hypersurfaces},
  pdfauthor={Michael Holst and Ari Stern},
  pdfsubject={MSC 2010: 65N30, 58A12},
  bookmarksopen=false
}

        \theoremstyle{plain} %--default
        \newtheorem{theorem}             {Theorem}  [section]
        \newaliascnt{lemma}{theorem}
        \newtheorem{lemma}      [lemma]{Lemma}
        \aliascntresetthe{lemma}
        \newaliascnt{corollary}{theorem}
        \newtheorem{corollary}  [corollary]{Corollary}
        \aliascntresetthe{corollary}

        \theoremstyle{definition}
        \newaliascnt{definition}{theorem}
        \newtheorem{definition} [definition]{Definition}
        \aliascntresetthe{definition}
        
        \newtheorem{example}    [theorem]{Example}

        \theoremstyle{remark}
        \newtheorem{remark}              {Remark}

\begin{document}

\title[Geometric variational crimes]{Geometric variational
  crimes:\\Hilbert complexes, finite element exterior calculus, and
  problems on hypersurfaces}

\author{Michael Holst}
\address{Department of Mathematics\\
University of California, San Diego\\
9500 Gilman Dr \#0112\\
La Jolla CA 92093-0112}
\email{\{mholst,astern\}@math.ucsd.edu}

\author{Ari Stern}


\begin{abstract}
  A recent paper of Arnold, Falk, and Winther
  [\emph{Bull. Amer. Math. Soc.} \textbf{47} (2010), 281--354] showed
  that a large class of mixed finite element methods can be formulated
  naturally on Hilbert complexes, where using a Galerkin-like
  approach, one solves a variational problem on a finite-dimensional
  subcomplex.  In a seemingly unrelated research direction, Dziuk
  [\emph{Lecture Notes in Math.}, vol.~1357 (1988), 142--155] analyzed
  a class of nodal finite elements for the Laplace--Beltrami equation
  on smooth $2$-surfaces approximated by a piecewise-linear
  triangulation; Demlow later extended this analysis [\emph{SIAM
    J. Numer. Anal.}, \textbf{47} (2009), 805--827] to $3$-surfaces,
  as well as to higher-order surface approximation.  In this article,
  we bring these lines of research together, first developing a
  framework for the analysis of variational crimes in abstract Hilbert
  complexes, and then applying this abstract framework to the setting
  of finite element exterior calculus on hypersurfaces.  Our framework
  extends the work of Arnold, Falk, and Winther to problems that
  violate their subcomplex assumption, allowing for the extension of
  finite element exterior calculus to approximate domains, most
  notably the Hodge--de~Rham complex on approximate manifolds.  As an
  application of the latter, we recover Dziuk's and Demlow's \emph{a
    priori} estimates for $2$- and $3$-surfaces, demonstrating that
  surface finite element methods can be analyzed completely within
  this abstract framework.  Moreover, our results generalize these
  earlier estimates dramatically, extending them from nodal finite
  elements for Laplace--Beltrami to mixed finite elements for the
  Hodge Laplacian, and from $2$- and $3$-dimensional hypersurfaces to
  those of arbitrary dimension.  By developing this analytical
  framework using a combination of general tools from differential
  geometry and functional analysis, we are led to a more geometric
  analysis of surface finite element methods, whereby the main results
  become more transparent.
\end{abstract}

%\date{September 13, 2011}

\maketitle

\clearpage

\tableofcontents

\section{Introduction}
\label{sec:intro}

The aim of this paper is to bring together three distinct ideas that
have influenced, in separate ways, the development and analysis of
geometric finite element methods for elliptic partial differential
equations.

The first idea is that of a \emph{variational crime}.  Suppose we have
a variational problem of the form: Find $ u \in V $ such that
\begin{equation}
  \label{eqn:variationalProblem}
  B \left( u, v \right) = F (v) , \quad \forall v \in V ,
\end{equation}
where $V$ is a Hilbert space, $ B \colon V \times V \rightarrow
\mathbb{R} $ is a bounded, coercive bilinear form, and $ F \in V ^\ast
$ is a bounded linear functional.  If $ V _h \subset V $ is a subspace
(usually finite-dimensional), then one can obtain an approximate
solution by solving the Galerkin variational problem: Find $ u _h \in
V _h $ such that
\begin{equation*}
  %\label{eqn:galerkin}
  B \left( u _h , v \right) = F (v) , \quad \forall v \in V _h .
\end{equation*}
This is the typical abstract setting for finite element methods.
However, for many problems of interest, especially finite element
methods on surfaces or on domains with curved boundaries, one cannot
efficiently compute the bilinear form $ B \left( \cdot , \cdot \right)
$ or the functional $ F \left( \cdot \right) $ on a subspace of $V$.
Instead, one must take an approximating space $ V _h \not\subset V $,
along with an approximate bilinear form $ B _h \colon V _h \times V _h
\rightarrow \mathbb{R} $ and functional $ F _h \in V _h ^\ast $, and
formulate the generalized Galerkin variational problem: Find $ u _h
\in V _h $ such that
\begin{equation}
  \label{eqn:generalizedGalerkin}
  B _h \left( u _h , v \right) = F _h (v) , \quad \forall v \in V _h .
\end{equation}
Such modifications to the original variational problem are called
``variational crimes.''  There is a well-understood framework for the
analysis of a large class of variational crimes, represented by the
\emph{Strang lemmas} \citep{Braess1997}: for instance, the first and
second Strang lemmas allow for the complete analysis of numerical
quadrature, the use of geometric modeling technology such as
isoparametric elements, and many other examples of variational crimes.

The emergence of \emph{surface finite elements} represents a second
distinct idea that has influenced the development of geometric finite
element methods.  The analysis of surface finite element methods,
which by construction are ``criminal'' methods, has required a more
sophisticated approach that exploits the specific nature of the crime
in order to obtain a satisfactory error analysis; this custom-tailored
analysis contrasts with the more general approach given by the Strang
lemmas.  The surface finite element research area was effectively
initiated with the 1988 article of \citet{Dziuk1988}, although there
is related work appearing about ten years earlier by
\citet{Nedelec1976}.  While there was some activity in the area during
the 1990s (cf.~\citep{Dziuk1991,DeDz1995}), beginning in 2001 there
was a tremendous expansion of research in the general area of surface
finite element methods, with many applications arising in material
science, biology, and astrophysics; examples
include~\citep{Holst2001,Christiansen2002,DeDz2003,DeDzEl2005,DzHu2006,DzEl2007,DeDz2007,Demlow2009}.

The third distinct idea that has had a major influence on the
development of geometric methods is that of \emph{mixed finite
  elements}, whose early success in areas such as computational
electromagnetics was later found to have surprising connections with
the calculus of exterior differential forms, including de~Rham
cohomology and Hodge
theory~\citep{Bossavit1988,Nedelec1980,Nedelec1986,GrKo2004}.  This
has culminated, very recently, in the powerful theory of \emph{finite
  element exterior calculus} developed
by~\citet{ArFaWi2006,ArFaWi2010}.  A key insight of the latter work,
from a functional-analytic point of view, is that a mixed variational
problem can be posed on a \emph{Hilbert complex}: a differential
complex of Hilbert spaces, in the sense of \citet{BrLe1992}.
Galerkin-type mixed methods are then obtained by solving the
variational problem on a finite-dimensional subcomplex.

In this article, we bring these lines of research together, first
developing a framework for the analysis of variational crimes in
abstract Hilbert complexes, and then applying this abstract framework
to the setting of finite element exterior calculus on hypersurfaces.
Our framework extends the work of \citet{ArFaWi2010} to problems that
violate their subcomplex assumption, allowing for the extension of
finite element exterior calculus to approximate domains, most notably
the Hodge--de~Rham complex on approximate manifolds.  As an
application of the latter, we recover Dziuk's \citep{Dziuk1988} and
Demlow's \citep{Demlow2009} \emph{a priori} estimates for $2$- and
$3$-surfaces, demonstrating that surface finite element methods can be
analyzed completely within this abstract framework.  Moreover, our
results generalize these earlier estimates dramatically, extending
them from nodal finite elements for Laplace--Beltrami to mixed finite
elements for the Hodge Laplacian, and from $2$- and $3$-dimensional
hypersurfaces to those of arbitrary dimension.  By developing this
analytical framework using a combination of general tools from
differential geometry and functional analysis, we are led to a more
geometric analysis of surface finite element methods, whereby the main
results become more transparent.

The remainder of the article is organized as follows.  In
\autoref{sec:hilbert}, we review the abstract framework of Hilbert
complexes, which plays a central role in the work of
\citet{ArFaWi2010} on finite element exterior calculus.  This includes
a brief introduction to Hilbert complexes and their morphisms, domain
complexes, Hodge decomposition, the Poincar\'{e} inequality, the Hodge
Laplacian, mixed variational problems, and approximation using Hilbert
subcomplexes.  In \autoref{sec:crimes}, we consider the approximation
of a Hilbert complex by a second complex, related to the first complex
through an injective morphism rather than through subcomplex
inclusion.  Since this morphism is not necessarily unitary (i.e.,
inner-product preserving), this allows the approximating complex to
have a different inner product, which only approximates that of the
original complex.  We develop some basic results for the pair of
complexes and the maps between them, and then prove error estimates
for generalized Galerkin-type approximations of solutions to
variational problems using the approximating complex; these estimates
generalize the results of \citet{ArFaWi2010} to ``external''
approximations.  Our results may be viewed as establishing
\emph{Strang-type lemmas} for approximating variational problems in
Hilbert complexes.  Finally, in \autoref{sec:diffForms}, we apply the
framework developed in \autoref{sec:crimes} to the Hodge--de~Rham
complex of differential forms on a compact, oriented Riemannian
manifold.  We first review Hodge--de~Rham theory, and then consider a
pair of Riemannian manifolds related by diffeomorphisms, establishing
estimates for the maps needed to apply the generalized Hilbert complex
approximation framework.  After reviewing the concept of a tubular
neighborhood, we then consider the specific case of Euclidean
hypersurfaces.  We subsequently show how the results of the previous
sections recover the analysis framework and \emph{a priori} estimates
of \citet{Dziuk1988,DeDz2007,Demlow2009}, and moreover extend their
results from scalar functions on 2- and 3-surfaces to general
$k$-forms on arbitrary dimensional hypersurfaces.  We also indicate
how our results generalize the \emph{a priori} estimates of
\citet{Dziuk1988,Demlow2009} from nodal finite element methods for the
Laplace--Beltrami operator to mixed finite element methods for the
Hodge Laplacian.

\section{Review of Hilbert complexes}
\label{sec:hilbert}

In this section, we quickly review the abstract framework of Hilbert
complexes, which forms the heart of the analysis in \citet{ArFaWi2010}
for mixed finite element methods.  Just as the space of $ L ^2 $
functions is a prototypical example of a Hilbert space, the
prototypical example of a Hilbert complex to keep in mind is the $ L
^2 $-de~Rham complex of differential forms.  (This example will be
discussed at greater length in \autoref{sec:diffForms}.)  After
stating the basic definitions, we will summarize some of the key
results from \citet{ArFaWi2010} on mixed variational problems and
their numerical approximation using Hilbert subcomplexes. The
interested reader may also refer to \citet{BrLe1992} for a
comprehensive treatment of Hilbert complexes from the viewpoint of
functional analysis.

\subsection{Basic definitions} Let us introduce the basic objects of
study, Hilbert complexes, and their morphisms.

\begin{definition}
  A \emph{Hilbert complex} $ \left( W , \mathrm{d} \right) $ consists
  of a sequence of Hilbert spaces $ W ^k $, along with closed,
  densely-defined linear maps $ \mathrm{d} ^k \colon V ^k \subset W ^k
  \rightarrow V ^{ k + 1 } \subset W ^{ k + 1 } $, possibly unbounded,
  such that $ \mathrm{d} ^k \circ \mathrm{d} ^{k-1} = 0 $ for each
  $k$.
  \begin{equation*}
    \xymatrix{  \cdots \ar[r] & V ^{ k - 1 } \ar[r]^-{ \mathrm{d} ^{k-1} }
      & V ^k \ar[r] ^-{ \mathrm{d} ^k } & V ^{ k + 1 } \ar[r] & \cdots }
  \end{equation*}
  This Hilbert complex is said to be \emph{bounded} if $ \mathrm{d} ^k
  $ is a bounded linear map from $ W ^k $ to $ W ^{ k + 1 } $ for each
  $k$, i.e., $ \left( W , \mathrm{d} \right) $ is a cochain complex in
  the category of Hilbert spaces.  It is said to be \emph{closed} if
  the image $ \mathrm{d} ^k V ^k $ is closed in $ W ^{ k + 1 } $ for
  each $k$.
\end{definition}

\begin{definition}
  \label{def:morphism}
  Given two Hilbert complexes, $ \left( W, \mathrm{d} \right) $ and $
  \left( W ^\prime , \mathrm{d} ^\prime \right) $, a \emph{morphism of
    Hilbert complexes} $ f \colon W \rightarrow W ^\prime $ consists
  of a sequence of bounded linear maps $ f ^k \colon W ^k \rightarrow
  W ^{\prime k} $ such that $ f ^k V ^k \subset V ^{ \prime k } $ and
  $ \mathrm{d} ^{\prime k} f ^k = f ^{ k + 1 } \mathrm{d} ^k $ for
  each $k$.  That is, the following diagram commutes:
  \begin{equation*}
    \xymatrix@=3em{
      \cdots \ar[r] & V ^k \ar[d]^{f ^k } \ar[r]^{\mathrm{d}^k} &
      V ^{k+1} \ar[d]^{f ^{k+1}} \ar[r] & \cdots \\
      \cdots \ar[r] & V ^{\prime k} \ar[r]^{\mathrm{d} ^{\prime k}}
      & V ^{\prime k+1} \ar[r] & \cdots 
    }
  \end{equation*}
\end{definition}

By analogy with cochain complexes, it is possible to define notions of
cocycles, coboundaries, harmonic forms, and cohomology spaces for
Hilbert complexes.

\begin{definition}
  Given a Hilbert complex $ \left( W, \mathrm{d} \right) $, the space
  of \emph{$k$-cocycles} is the kernel $ \mathfrak{Z} ^k = \ker
  \mathrm{d} ^k $, the space of \emph{$k$-coboundaries} is the image $
  \mathfrak{B} ^k = \mathrm{d} ^{ k - 1 } V ^{ k - 1 } $, the {\em
    $k$th harmonic space} is the intersection $ \mathfrak{H} ^k =
  \mathfrak{Z} ^k \cap \mathfrak{B} ^{k \perp _W} $, and the
  \emph{$k$th reduced cohomology space} is the quotient $ \mathfrak{Z}
  ^k / \overline{ \mathfrak{B} ^k } $.  When $ \mathfrak{B} ^k $ is
  closed, $ \mathfrak{Z} ^k / \mathfrak{B} ^k $ is simply called the
  \emph{$k$th cohomology space}, and is identical to reduced
  cohomology.
\end{definition}

\begin{remark}
  One can show that the harmonic space $ \mathfrak{H} ^k $ is
  isomorphic to the reduced cohomology space $ \mathfrak{Z} ^k /
  \overline{ \mathfrak{B} ^k } $.  For a closed complex, this is
  identical to the usual cohomology space $ \mathfrak{Z} ^k /
  \mathfrak{B} ^k $, since $ \mathfrak{B} ^k $ is closed for each $k$.
\end{remark}

\begin{definition}
  Given a morphism of Hilbert complexes $ f \colon W \rightarrow W
  ^\prime $, the \emph{induced map on (reduced) cohomology} is defined
  by $ [z] \mapsto [fz] $, where $[z]$ denotes the (reduced)
  cohomology class of the cocycle $z$.
\end{definition}

In general, the differentials $ \mathrm{d} ^k $ of a Hilbert complex
may be unbounded linear maps.  However, given an arbitrary Hilbert
complex $ \left( W, \mathrm{d} \right) $, it is always possible to
construct a bounded complex having the same domains and maps, as
follows.

\begin{definition}
  Given a Hilbert complex $ \left( W, \mathrm{d} \right) $, the {\em
    domain complex} $ \left( V, \mathrm{d} \right) $ consists of the
  domains $ V ^k \subset W ^k $, endowed with the graph inner product
\begin{equation*}
  \left\langle u, v \right\rangle _{ V ^k } = \left\langle u , v
  \right\rangle _{ W ^k } + {\langle \mathrm{d} ^k u, \mathrm{d}
    ^k v \rangle} _{ W ^{ k + 1 } } .
\end{equation*}
\end{definition}

\begin{remark}
  Since $ \mathrm{d} ^k $ is a closed map, each $ V ^k $ is closed
  with respect to the norm induced by the graph inner product.  Also,
  each map $ \mathrm{d} ^k $ is bounded, since
  \begin{equation*}
    {\lVert \mathrm{d} ^k v \rVert} _{V^{k+1}} = {\lVert \mathrm{d}
      ^k v \rVert} _{W^{k+1}} \leq \left\lVert v \right\rVert _{ W ^k
    } + {\lVert \mathrm{d} ^k v \rVert} _{W^{k+1}} = \left\lVert
      v \right\rVert _{ V ^k } .
  \end{equation*}
  Thus, the domain complex is a bounded Hilbert complex; moreover, it
  is a closed complex if and only if $ \left( W, \mathrm{d} \right) $
  is closed.
\end{remark}

For the remainder of the paper, we will follow the simplified notation
used by \citet{ArFaWi2010}: the $W$-inner product and norm will be
written simply as $ \left\langle \cdot , \cdot \right\rangle $ and $
\left\lVert \cdot \right\rVert $, without subscripts, while the
$V$-inner product and norm will be written explicitly as $
\left\langle \cdot , \cdot \right\rangle _V $ and $ \left\lVert \cdot
\right\rVert _V $.

\subsection{Hodge decomposition and Poincar\'e inequality}

The Helmholtz decomposition states that a rapidly-decaying vector
field on $ \mathbb{R}^3 $ can be decomposed into curl-free and
divergence-free components, i.e., the vector field can be written as
the sum of the gradient of a scalar potential and the curl of a vector
potential.  For differential forms, this is generalized by the Hodge
decomposition, which states that any differential form can be written
as a sum of exact, coexact, and harmonic components.  Here, we recall
an even further generalization of the Hodge decomposition to arbitrary
Hilbert complexes; this immediately gives rise to an abstract version
of the Poincar\'e inequality, which will be crucial to much of the
later analysis.

Following \citet{BrLe1992}, we can decompose each space $ W ^k $ in
terms of orthogonal subspaces,
\begin{equation*}
  W ^k = \mathfrak{Z}  ^k \oplus  \mathfrak{Z}  ^{ k \perp _W } =
  \mathfrak{Z}  ^k \cap \bigl( \overline{ \mathfrak{B}  ^k } \oplus
  \mathfrak{B} ^{k \perp _W } \bigr) \oplus \mathfrak{Z}  ^{ k \perp _W } =
  \overline{ \mathfrak{B}  ^k } \oplus \mathfrak{H}  ^k \oplus
  \mathfrak{Z}  ^{ k \perp _W } ,
\end{equation*}
where the final expression is known as the \emph{weak Hodge
  decomposition}.  
For the domain complex $ \left( V,
  \mathrm{d} \right) $, the spaces $ \mathfrak{Z} ^k $, $ \mathfrak{B}
^k $, and $ \mathfrak{H} ^k $ are the same as for $ \left( W ,
  \mathrm{d} \right) $, and consequently we get the decomposition
\begin{equation*}
  V ^k = \overline{ \mathfrak{B}  ^k } \oplus \mathfrak{H}  ^k \oplus
  \mathfrak{Z} ^{ k \perp _V } ,
\end{equation*}
where $ \mathfrak{Z} ^{ k \perp _V } = \mathfrak{Z} ^{ k \perp _W }
\cap V ^k $.  In particular, if $ \left( W, \mathrm{d} \right) $ is a
closed Hilbert complex, then the image $ \mathfrak{B} ^k $ is a closed
subspace, so we have the \emph{strong Hodge decomposition}
\begin{equation*}
  W ^k = \mathfrak{B}  ^k \oplus \mathfrak{H}  ^k \oplus
  \mathfrak{Z}  ^{ k \perp _W } ,
\end{equation*}
and likewise for the domain complex,
\begin{equation*}
  V ^k = \mathfrak{B}  ^k \oplus \mathfrak{H}  ^k \oplus
  \mathfrak{Z} ^{ k \perp _V } .
\end{equation*}
From here on, following the notation of \citet{ArFaWi2010}, we will
simply write $ \mathfrak{Z} ^{ k \perp} $ in place of $ \mathfrak{Z}
^{ k \perp _V } $ when there can be no confusion.

\begin{lemma}[Abstract Poincar\'e Inequality]
  \label{lem:poincareInequality}
  If $ \left( V, \mathrm{d} \right) $ is a bounded, closed Hilbert
  complex, then there exists a constant $ c _P $ such that
  \begin{equation*}
    \left\lVert v \right\rVert _V \leq c _P {\lVert \mathrm{d} ^k v
      \rVert} _V , \quad \forall v \in \mathfrak{Z}  ^{ k \perp } .
  \end{equation*}
\end{lemma}

\begin{proof}
  The map $ \mathrm{d} ^k $ is a bounded bijection from $ \mathfrak{Z}
  ^{ k \perp } $ to $ \mathfrak{B} ^{ k + 1 } $, which are both closed
  subspaces, so the result follows immediately by applying the bounded
  inverse theorem.
\end{proof}

\begin{corollary}
  If $ \left( V, \mathrm{d} \right) $ is the domain complex of a
  closed Hilbert complex $ \left( W, \mathrm{d} \right) $, then
  \begin{equation*}
    \left\lVert v \right\rVert _V \leq c _P {\lVert \mathrm{d} ^k v
      \rVert} , \quad \forall v \in \mathfrak{Z}  ^{ k \perp } .
  \end{equation*}
\end{corollary}

We close this subsection by defining the dual complex of a Hilbert
complex, and recalling how the Hodge decomposition can be interpreted in
terms of this complex.

\begin{definition}
  Given a Hilbert complex $ \left( W, \mathrm{d} \right) $, the {\em
    dual complex} $ \left( W ^\ast , \mathrm{d} ^\ast \right) $
  consists of the spaces $ W _k ^\ast = W ^k $, and adjoint operators
  $ \mathrm{d} _k ^\ast = \left( \mathrm{d} ^{ k - 1 } \right) ^\ast
  \colon V _k ^\ast \subset W _k ^\ast \rightarrow V _{ k - 1 } ^\ast
  \subset W _{ k - 1 } ^\ast $.
\begin{equation*}
  \xymatrix{
    \cdots & \ar[l]  V _{k-1} ^\ast & \ar[l] _-{ \mathrm{d} _k ^\ast }
    V _k ^\ast & \ar[l] _-{ \mathrm{d} _{ k + 1 } ^\ast }  V _{ k
      + 1 } ^\ast & \ar[l] \cdots 
  } 
\end{equation*}
\end{definition}

\begin{remark}
  Since the arrows in the dual complex point in the opposite
  direction, this is a Hilbert chain complex rather than a cochain
  complex.  (The chain property $ \mathrm{d} _k ^\ast \circ \mathrm{d}
  _{ k + 1 } ^\ast = 0 $ follows easily from the cochain property $
  \mathrm{d} ^k \circ \mathrm{d} ^{ k - 1 } = 0 $.) Accordingly, we
  can define the \emph{$k$-cycles} $ \mathfrak{Z} _k ^\ast = \ker
  \mathrm{d} _k ^\ast = \mathfrak{B} ^{ k \perp _W } $ and \emph{$ k
    $-boundaries} $ \mathfrak{B} _k ^\ast = \mathrm{d} _{ k + 1 }
  ^\ast V ^\ast _k $.  The $k$th harmonic space can then be rewritten
  as $ \mathfrak{H} ^k = \mathfrak{Z} ^k \cap \mathfrak{Z} _k ^\ast $;
  we also have $ \mathfrak{Z} ^k = \mathfrak{B} _k ^{\smash{\ast \perp
      _W }} $, and thus $ \mathfrak{Z} ^{k \perp _W } = \overline{
    \mathfrak{B} _k ^\ast } $.  Therefore, the weak Hodge
  decomposition can be written as
  \begin{equation*}
    W ^k = \overline{ \mathfrak{B}  ^k } \oplus \mathfrak{H}  ^k \oplus
    \overline{ \mathfrak{B}  _k ^\ast } ,
  \end{equation*}
  and in particular, for a closed Hilbert complex, the strong Hodge
  decomposition now becomes
  \begin{equation*}
    W ^k = \mathfrak{B}  ^k \oplus \mathfrak{H}  ^k \oplus
    \mathfrak{B}  _k ^\ast .
  \end{equation*}
\end{remark}

\subsection{The abstract Hodge Laplacian and mixed variational
  problem}

To obtain a ``mixed version'' of the familiar Poisson equation $ -
\Delta u = f $ for scalar functions, we now follow \citet{ArFaWi2010}
in defining an abstract version of the Hodge Laplacian for Hilbert
complexes.  The \emph{abstract Hodge Laplacian} is the operator $ L =
\mathrm{d} \mathrm{d} ^\ast + \mathrm{d} ^\ast \mathrm{d} $, which is
an unbounded operator $ W ^k \rightarrow W ^k $ with domain
\begin{equation*}
  D _L = \left\{ u \in V ^k \cap V _k ^\ast \;\middle\vert\;
    \mathrm{d} u \in V _{ k + 1 } ^\ast ,\ \mathrm{d} ^\ast u \in V ^{
      k - 1 } \right\} .
\end{equation*}
If $u \in D _L $ solves $ L u = f $, then it satisfies the variational
principle
\begin{equation*}
  \left\langle \mathrm{d} u , \mathrm{d} v \right\rangle +
  \left\langle \mathrm{d} ^\ast u , \mathrm{d} ^\ast v \right\rangle =
  \left\langle f, v \right\rangle , \quad \forall v \in V ^k \cap V _k
  ^\ast .
\end{equation*}
However, as noted by \citet{ArFaWi2010}, there are some difficulties
in using this variational principle for a finite element
approximation.  First, it may be difficult to construct finite
elements for the space $ V ^k \cap V _k ^\ast $.  A second concern is
the well-posedness of the problem.  If we take any harmonic test
function $ v \in \mathfrak{H} ^k $, then the left-hand side vanishes,
so $ \left\langle f, v \right\rangle = 0 $; hence, a solution only
exists if $ f \perp \mathfrak{H} ^k $.  Furthermore, for any $ q \in
\mathfrak{H} ^k = \mathfrak{Z} ^k \cap \mathfrak{Z} _k ^\ast $, we
have $ \mathrm{d} q = 0 $ and $ \mathrm{d} ^\ast q = 0 $; therefore,
if $u$ is a solution, then so is $ u + q $.

To avoid these existence and uniqueness issues, one can define instead
the following mixed variational problem: Find $ \left( \sigma, u, p
\right) \in V ^{ k - 1 } \times V ^k \times \mathfrak{H} ^k $
satisfying
\begin{equation}
  \label{eqn:mixedProblem}
  \begin{alignedat}{2}
    \left\langle \sigma , \tau \right\rangle - \left\langle u ,
      \mathrm{d} \tau \right\rangle &= 0, &\quad \forall \tau &\in V
    ^{ k - 1 } ,\\
    \left\langle \mathrm{d} \sigma, v \right\rangle + \left\langle
      \mathrm{d} u, \mathrm{d} v \right\rangle + \left\langle p, v
    \right\rangle &= \left\langle f, v \right\rangle , &\quad \forall
    v &\in V ^k ,\\
    \left\langle u, q \right\rangle &= 0 , &\quad \forall q &\in
    \mathfrak{H} ^k .
  \end{alignedat}
\end{equation}
Here, the first equation implies that $ \sigma = \mathrm{d} ^\ast u $,
which weakly enforces the condition $ u \in V ^k \cap V _k ^\ast
$. Next, the second equation incorporates the additional term $
\left\langle p, v \right\rangle $, which allows for solutions to exist
even when $f \not\perp \mathfrak{H} ^k $.  Finally, the third equation
fixes the issue of non-uniqueness by requiring $ u \perp \mathfrak{H}
^k $.  The following result establishes the well-posedness of the
problem \eqref{eqn:mixedProblem}.

\begin{theorem}[\citet{ArFaWi2010}, Theorem 3.1]
  Let $ \left( W, \mathrm{d} \right) $ be a closed Hilbert complex
  with domain complex $ \left( V, \mathrm{d} \right) $.  The mixed
  formulation of the abstract Hodge Laplacian is well-posed.  That is,
  for any $ f \in W ^k $, there exists a unique $ \left( \sigma, u, p
  \right) \in V ^{ k - 1 } \times V ^k \times \mathfrak{H} ^k $
  satisfying \eqref{eqn:mixedProblem}.  Moreover,
  \begin{equation*}
    \left\lVert \sigma \right\rVert _V + \left\lVert u \right\rVert _V
    + \left\lVert p \right\rVert \leq c \left\lVert f \right\rVert ,
  \end{equation*}
  where $c$ is a constant depending only on the Poincar\'e constant $
  c _P $ in \autoref{lem:poincareInequality}.
\end{theorem}

To prove this, one observes that \eqref{eqn:mixedProblem} can be
rewritten as a standard variational problem
\eqref{eqn:variationalProblem} on the space $ V ^{ k - 1 } \times V ^k
\times \mathfrak{H} ^k $, with the bilinear form
\begin{equation*}
  B \left( \sigma  , u , p ; \tau , v , q \right)
  = \left\langle \sigma  , \tau  \right\rangle  - \left\langle u
     , \mathrm{d} \tau \right\rangle  \\
  + \left\langle \mathrm{d} \sigma  , v  \right\rangle  +
  \left\langle \mathrm{d} u  , \mathrm{d} v  \right\rangle  +
  \left\langle p  , v  \right\rangle  - \left\langle u  , q 
  \right\rangle 
\end{equation*} 
and functional $ F \left( \tau, v, q \right) = \left\langle f, v
\right\rangle $.  The well-posedness then follows immediately from the
following theorem, which establishes the inf-sup condition for the
bilinear form $B$.

\begin{theorem}[\citet{ArFaWi2010}, Theorem 3.2]
  \label{thm:inf-sup}
  Let $ \left( W, \mathrm{d} \right) $ be a closed Hilbert complex
  with domain complex $ \left( V , \mathrm{d} \right) $.  There exists
  a constant $ \gamma > 0 $, depending only on the constant $ c _P $
  in the Poincar\'e inequality (\autoref{lem:poincareInequality}),
  such that for any $ \left( \sigma, u , p \right) \in V ^{ k - 1 }
  \times V ^k \times \mathfrak{H}  ^k $, there exists $ \left( \tau ,
    v, q \right) \in V ^{ k - 1 } \times V ^k \times \mathfrak{H}  ^k
  $ with
  \begin{equation*}
    B \left( \sigma, u, p; \tau , v, q \right) \geq \gamma \left(
      \left\lVert \sigma \right\rVert _V + \left\lVert u \right\rVert
      _V + \left\lVert p \right\rVert \right) \left( \left\lVert \tau
      \right\rVert _V + \left\lVert v \right\rVert _V + \left\lVert q
      \right\rVert \right) .
  \end{equation*}
\end{theorem}

From the well-posedness result, it follows that there exists a bounded
solution operator $ K \colon W ^k \rightarrow W ^k $ defined by $ K f
= u $.

\subsection{Approximation by a subcomplex}

In order to obtain approximate numerical solutions to the mixed
variational problem \eqref{eqn:mixedProblem}, \citet{ArFaWi2010}
suppose that one is given a (finite-dimensional) subcomplex $ V _h
\subset V $ of the domain complex: that is, $ V _h ^k \subset V ^k $
is a Hilbert subspace for each $k$, and the inclusion mapping $ i _h
\colon V _h \hookrightarrow V $ is a morphism of Hilbert complexes.
By analogy with the Galerkin method, one can then consider the mixed
variational problem on the subcomplex: Find $ \left( \sigma _h , u _h
  , p _h \right) \in V _h ^{ k - 1 } \times V _h ^k \times
\mathfrak{H} _h ^k $ satisfying
\begin{equation}
  \label{eqn:subcomplexProblem}
  \begin{alignedat}{2}
    \left\langle \sigma _h , \tau \right\rangle - \left\langle u _h ,
      \mathrm{d} \tau \right\rangle &= 0, &\quad \forall \tau &\in V
    _h
    ^{ k - 1 } ,\\
    \left\langle \mathrm{d} \sigma _h , v \right\rangle + \left\langle
      \mathrm{d} u _h , \mathrm{d} v \right\rangle + \left\langle p _h
      , v \right\rangle &= \left\langle f, v \right\rangle , &\quad
    \forall
    v &\in V _h  ^k ,\\
    \left\langle u _h , q \right\rangle &= 0 , &\quad \forall q &\in
    \mathfrak{H} _h ^k .
  \end{alignedat}
\end{equation}

For the error analysis of this method, one more crucial assumption
must be made: that there exists some Hilbert complex ``projection'' $
\pi _h \colon V \rightarrow V _h $.  We put ``projection'' in quotes
because this need not be the actual orthogonal projection $ i _h ^\ast
$ with respect to the inner product; indeed, that projection is not
generally a morphism of Hilbert complexes, since it may not commute
with the differentials.  However, the map $ \pi _h $ is $V$-bounded,
surjective, and idempotent.  It follows, then, that although it does
not satisfy the optimality property of the true projection, it does
still satisfy a \emph{quasi-optimality} property, since
\begin{equation*}
  \left\lVert u - \pi _h u  \right\rVert _V = \inf _{ v \in V _h }
  \left\lVert \left( I - \pi _h \right) \left( u - v \right)
  \right\rVert _V \leq \left\lVert I - \pi _h \right\rVert \inf _{ v
    \in V _h } \left\lVert u - v \right\rVert _V ,
\end{equation*}
where the first step follows from the idempotence of $ \pi _h $, i.e.,
$ \left( I - \pi _h \right) v = 0 $ for all $ v \in V _h $.  With this
framework in place, the following error estimate can be established.

\begin{theorem}[\citet{ArFaWi2010}, Theorem 3.9]
  \label{thm:subcomplex}
  Let $ \left( V _h , \mathrm{d} \right) $ be a family of subcomplexes
  of the domain complex $ \left( V, \mathrm{d} \right) $ of a closed
  Hilbert complex, parametrized by $h$ and admitting uniformly
  $V$-bounded cochain projections, and let $ \left( \sigma, u, p
  \right) \in V ^{ k - 1 } \times V ^k \times \mathfrak{H} ^k $ be the
  solution of \eqref{eqn:mixedProblem} and $ \left( \sigma _h , u _h ,
    p _h \right) \in V _h ^{ k - 1 } \times V _h ^k \times
  \mathfrak{H} _h ^k $ the solution of problem
  \eqref{eqn:subcomplexProblem}.  Then
  \begin{multline*}
    \left\lVert \sigma - \sigma _h \right\rVert _V + \left\lVert u - u
      _h \right\rVert _V + \left\lVert p - p _h \right\rVert \\
    \leq C \bigl( \inf _{ \tau \in V _h ^{ k - 1 } } \left\lVert
      \sigma - \tau \right\rVert _V + \inf _{ v \in V _h ^k }
    \left\lVert u - v \right\rVert _V + \inf _{ q \in V _h ^k }
    \left\lVert p - q \right\rVert _V + \mu \inf _{ v \in V _h ^k }
    \left\lVert P _{ \mathfrak{B} } u - v \right\rVert _V \bigr) ,
  \end{multline*}
  where $ \mu = \mu _h ^k = \displaystyle\sup _{ \substack{r \in
      \mathfrak{H} ^k \\ \left\lVert r \right\rVert = 1 } }
  \left\lVert \left( I - \pi _h ^k \right) r \right\rVert $.
\end{theorem}

Therefore, if $ V _h $ is pointwise approximating, in the sense that $
\inf _{ v \in V _h } \left\lVert u - v \right\rVert \rightarrow 0 $ as
$ h \rightarrow 0 $ for every $ u \in V $, then the numerical solution
converges to the exact solution.

\section{Analysis of variational crimes}
\label{sec:crimes}

In this section, we extend the results of \citet{ArFaWi2010},
summarized in the previous section, by removing the requirement for $
V _h $ to be a subcomplex of $V$.  The key point of departure is in
the map $ i _h \colon V _h \hookrightarrow V $; rather than being an
inclusion, we require only that it is an injective morphism of Hilbert
complexes, with the property that $ \pi _h \circ i _h $ is the
identity.  (The latter requirement simply corresponds to the earlier
condition that $ \pi _h $ be idempotent in the case of subcomplexes.)
After stating some basic results about complexes equipped with such
maps, we develop error estimates for the mixed variational problem and
eigenvalue problem on $ V _h $.  These estimates contain two
additional error terms, in addition to those in the analysis of
\citet{ArFaWi2010}.  These extra terms, analogous to those in the
Strang lemmas for generalized Galerkin methods, measure the
``severity'' of two variational crimes: first, how well the right-hand
side $ i _h f _h $ approximates $f$; and second, the extent to which $
i _h $ fails to be unitary.

\subsection{Approximation by an arbitrary complex}
In order to approximate a Hilbert complex $ \left( W, \mathrm{d}
\right) $, suppose we have another Hilbert complex $ \left( W _h ,
  \mathrm{d} _h \right) $, along with a pair of morphisms: an
injection $ i _h \colon W _h \hookrightarrow W $ and a projection $
\pi _h \colon W \rightarrow W _h $, such that $ \pi _h ^k \circ i _h
^k $ is the identity on $ W _h ^k $ for each $k$.  Recall that, by
\autoref{def:morphism} of a Hilbert complex morphism, the maps $ i _h
^k $ and $ \pi _h ^k $ must be bounded for each $k$.  The
relationships among the domains and maps are illustrated in the
following diagram:
\begin{equation*}
  \xymatrix@=3em{
    \cdots \ar[r] & V ^k \ar@<.3em>[d]^{\pi _h ^k }
    \ar[r]^{\mathrm{d}^k}&   V ^{k+1}
    \ar@<.3em>[d]^{\pi _h ^{k+1}} \ar[r] & \cdots \phantom{.}\\
    \cdots \ar[r] & V _h ^k \ar@<.3em>[u]^{i _h ^k} \ar[r]^{\mathrm{d}
      _h ^k}
    &   V _h ^{k+1} \ar@<.3em>[u]^{i _h ^{k+1}} \ar[r] & \cdots .
  } 
\end{equation*}
\citet{ArFaWi2010} consider the case where $ W _h \subset W $ is a
subcomplex, and $ i _h $ is the inclusion of $ W _h $ into $W$.  In
this special case, $ i _h $ is unitary (i.e., an isometry), since for
all $ u, v \in W _h ^k $, we have $ \left\langle i _h u , i _h v
\right\rangle = \left\langle u, v \right\rangle = \left\langle u , v
\right\rangle _h $.  Indeed, if $ i _h $ is unitary, then we can
simply identify $ W _h $ with the subcomplex $ i _h W _h \subset W$.
More generally, though, we will consider cases where $ W _h
\not\subset W $, and where $ i _h $ is not necessarily unitary.

We begin by demonstrating some basic facts about these approximations.

\begin{theorem}
  If $ \left( W , \mathrm{d} \right) $ is a bounded Hilbert complex,
  then so is $ \left( W _h , \mathrm{d} _h \right) $.
\end{theorem}

\begin{proof}
  $ \bigl\lVert \mathrm{d} _h ^k \bigr\rVert = \bigl\lVert \pi _h
  ^{k+1} i _h ^{ k + 1 } \mathrm{d} _h ^k \bigr\rVert = \bigl\lVert
  \pi _h ^{ k + 1 } \mathrm{d} ^k i _h ^k \bigr\rVert \leq \bigl\lVert
  \pi _h ^{ k + 1 } \bigr\rVert \bigl\lVert \mathrm{d} ^k \bigr\rVert
  \bigl\lVert i _h ^k \bigr\rVert < \infty $.
\end{proof}

\begin{theorem}
  If $ \left( W , \mathrm{d} \right) $ is a closed Hilbert complex,
  then so is $ \left( W _h , \mathrm{d} _h \right) $.
\end{theorem}

\begin{proof}
  Assume that $ \left( W, \mathrm{d} \right) $ is closed, so that each
  coboundary space $ \mathfrak{B} ^k $ is closed in $ W ^k $.  Now,
  since $ i _h $ is a morphism, if $ v _h \in \mathfrak{B} _h ^k $
  then $ i _h v _h \in \mathfrak{B} ^k $, so $ \mathfrak{B} _h ^k
  \subset i _h ^{-1} \mathfrak{B} ^k $.  Conversely, since $ \pi _h $
  is a morphism, if $ i _h v _h \in \mathfrak{B} ^k $ then $ v _h =
  \pi _h i _h v _h \in \mathfrak{B} _h ^k $, so $ i _h ^{-1}
  \mathfrak{B} ^k \subset \mathfrak{B} _h ^k $.  Therefore, $
  \mathfrak{B} _h ^k = i _h ^{-1} \mathfrak{B} ^k $, and since $ i _h
  $ is bounded (and hence continuous), it follows that $ \mathfrak{B}
  _h ^k $ is closed.
\end{proof}

Since $ \pi _h ^k \circ i _h ^k = \mathrm{id} _{ W _h ^k } $, this
composition induces the identity map on the reduced cohomology space $
\mathfrak{Z} _h ^k / \overline{ \mathfrak{B} _h ^k } $; thus $ i _h $
induces an injection on reduced cohomology, while $ \pi _h $ induces a
surjection.  We now show that, given a certain approximation condition
on the harmonic spaces $ \mathfrak{H} ^k $, these induced maps are in
fact isomorphisms (which are inverses of one another, since their
composition is the identity).

\begin{theorem}
  Let $ \left( W, \mathrm{d} \right) $ and $ \left( W _h , \mathrm{d}
    _h \right) $ be Hilbert complexes, with morphisms $ i _h \colon W
  _h \hookrightarrow W $ and $ \pi _h \colon W \rightarrow W _h
  $ such that $ \pi _h ^k \circ i _h ^k = \mathrm{id} _{ W _h ^k } $
  for each $k$.  If, for all $k$,
  \begin{equation*}
    \bigl\lVert q - i _h ^k \pi _h ^k q \bigr\rVert < \left\lVert q
    \right\rVert , \quad \forall  q \in \mathfrak{H}  ^k,\ q \neq 0 ,
  \end{equation*}
  then $ \pi _h $ (and thus $ i _h $) induces an isomorphism on
  the reduced cohomology spaces.
\end{theorem}

\begin{proof}
  Since $ \pi _h $ induces a surjection on reduced cohomology, it
  suffices to show that this is also an injection.  That is, given $ z
  \in \mathfrak{Z} ^k $ with $ \pi _h z \in \overline{ \mathfrak{B} _h
    ^k } $, we must demonstrate that $ z \in \overline{ \mathfrak{B}
    ^k } $.  Using the weak Hodge decomposition, write $ z = q + b $,
  where $ q \in \mathfrak{H} ^k $ and $ b \in \overline{ \mathfrak{B}
    ^k } $.  By assumption, $ \pi _h z \in \overline{ \mathfrak{B} _h
    ^k } $, and since $ \pi _h $ is a morphism, $ \pi _h b \in
  \overline{ \mathfrak{B} _h ^k } $ as well.  Thus, $ \pi _h q = \pi
  _h z - \pi _h b \in \overline{ \mathfrak{B} _h ^k } $, and since $ i
  _h $ is also a morphism, $ i _h \pi _h q \in \overline{ \mathfrak{B}
    ^k } \perp \mathfrak{H} ^k $.  Therefore, $ i _h \pi _h q \perp q
  $, which implies that $q$ violates the inequality above, so we must
  have $ q = 0 $ and hence $ z \in \overline{ \mathfrak{B} ^k } $.
\end{proof}

\begin{corollary}
  If $ \left( W, \mathrm{d} \right) $ and $ \left( W _h , \mathrm{d}
    _h \right) $ are closed Hilbert complexes, with morphisms $ \pi _h
  $ and $ i _h $ satisfying the above assumptions, then $ \pi _h $
  (and thus $ i _h $) induces an isomorphism on cohomology.
\end{corollary}

\begin{remark}
  This result is slightly more general than \citet[Theorem
  3.4]{ArFaWi2010}, which only treated the case of a bounded, closed
  Hilbert complex.  However, the proof is essentially identical.
\end{remark}

Next, suppose that $ \left( V , \mathrm{d} \right) $ and $ \left( V _h
  , \mathrm{d} _h \right) $ are bounded, closed Hilbert complexes; for
example, they may be the domain complexes corresponding, respectively,
to closed complexes $ \left( W, \mathrm{d} \right) $ and $ \left( W _h
  , \mathrm{d} _h \right) $.  We now show that the Poincar\'e
inequality for $ \left( V _h , \mathrm{d} _h \right) $ can be written
entirely in terms of the Poincar\'e constant for $ \left( V,
  \mathrm{d} \right) $, denoted by $ c _P $, along with the operator
norms of $ i _h $ and $ \pi _h $.

\begin{theorem}
  \label{thm:discretePoincare}
  Let $ \left( V, \mathrm{d} \right) $ and $ \left( V _h , \mathrm{d}
    _h \right) $ be bounded, closed Hilbert complexes, with morphisms
  $ i _h \colon V _h \hookrightarrow V $ and $ \pi _h \colon V
  \rightarrow V _h $ such that $ \pi _h ^k \circ i _h ^k =
  \mathrm{id} _{ V _h ^k } $ for each $k$.  Then
  \begin{equation*}
    \left\lVert v _h \right\rVert _{V _h} \leq c _P  \bigl\lVert \pi _h
    ^k \bigr\rVert \bigl\lVert i _h ^{k+1} \bigr\rVert \left\lVert
      \mathrm{d} _h v _h \right\rVert _{V_h} ,  \quad \forall v _h \in
    \mathfrak{Z}  _h ^{ k \perp}.
  \end{equation*}
\end{theorem}

\begin{proof}
  Given $ v _h \in \mathfrak{Z} _h ^{ k \perp } $, let $ z \in
  \mathfrak{Z} ^{k \perp} $ be the unique element such that $
  \mathrm{d} z = \mathrm{d} i _h v _h = i _h \mathrm{d} _h v _h $.  Then,
  applying the abstract Poincar\'e inequality on $V$,
  \begin{equation*}
    \left\lVert z \right\rVert _V \leq c _P \left\lVert \mathrm{d} z
    \right\rVert _V = c _P \left\lVert i _h \mathrm{d} _h v _h 
    \right\rVert _V \leq c _P
    \bigl\lVert i _h ^{ k + 1 } \bigr\rVert \left\lVert \mathrm{d} _h v _h 
    \right\rVert _{V_h} .
  \end{equation*}
  It now suffices to show $ \left\lVert v _h \right\rVert _{ V _h } \leq
  \bigl\lVert \pi _h ^k \bigr\rVert \left\lVert z \right\rVert _V $.
  Observe that $ v _h - \pi _h z \in V _h ^k $, and furthermore,
  \begin{equation*}
    \mathrm{d} _h \pi _h z = \pi _h \mathrm{d} z = \pi _h i _h
    \mathrm{d} _h v _h = \mathrm{d} _h v _h ,
  \end{equation*}
  so $ v _h - \pi _h z \in \mathfrak{Z} _h ^k \perp v _h $.  Therefore,
  \begin{equation*}
    \left\lVert v _h \right\rVert _{ V _h } ^2 = \left\langle v _h, \pi _h z
    \right\rangle _{ V _h } + \left\langle v _h , v _h - \pi _h z
    \right\rangle _{ V _h } = \left\langle v _h , \pi _h z \right\rangle _{
      V _h } \leq \left\lVert v _h \right\rVert _{ V _h } \bigl\lVert \pi
    _h ^k \bigr\rVert  \left\lVert z \right\rVert _V ,
  \end{equation*}
  and the result follows.
\end{proof}

\begin{corollary}
  If $ \left( V, \mathrm{d} \right) $ and $ \left( V _h , \mathrm{d}
    _h \right) $ are the domain complexes corresponding, respectively,
  to closed Hilbert complexes $ \left( W, \mathrm{d} \right) $ and $
  \left( W _h , \mathrm{d} _h \right) $, then
  \begin{equation*}
    \left\lVert v _h \right\rVert _{V _h} \leq c _P  \bigl\lVert \pi _h
    ^k \bigr\rVert \bigl\lVert i _h ^{k+1} \bigr\rVert \left\lVert
      \mathrm{d} _h v _h \right\rVert _h ,  \quad \forall v _h \in
    \mathfrak{Z}  _h ^{ k \perp}.
  \end{equation*}
\end{corollary}

Finally, given the importance of the projection morphism $ \pi _h $ in
finite element exterior calculus, we now prove a short but useful
result on the existence of such projections.  In particular, the next
theorem states how a projection morphism on another complex, $ W
^\prime $, can be ``pulled back'' to obtain one on $W$, as pictured in
the following diagram:
\begin{equation*}
  \xymatrix{
    W \ar[rr] ^f  \ar@<.3em>@{-->}[rd] ^{ \pi _h } & & W ^\prime
    \ar@<-.3em>[ld] _{ \pi _h ^\prime } \\
    & \ar@<.3em>[ul] ^{ i _h } W _h \ar@<-.3em>[ur] _{ i _h ^\prime 
    }   &
  } .
\end{equation*} 
In \autoref{sec:diffForms}, this will allow us to obtain a projection
morphism for the de~Rham complex on a manifold, by pulling back the
usual projection defined on its piecewise-linear triangulation.

\begin{theorem}
  \label{thm:projection}
  Let $ \left( W, \mathrm{d} \right) $ and $ \left( W _h , \mathrm{d}
    _h \right) $ be Hilbert complexes with an injection morphism $ i
  _h \colon W _h \hookrightarrow W $.  Suppose there exists another
  complex $ \left( W ^\prime, \mathrm{d} ^\prime \right) $ and a
  morphism $ f \colon W \rightarrow W ^\prime $, such that $ i _h
  ^\prime = f \circ i _h \colon W _h \hookrightarrow W ^\prime $ is
  injective and has a corresponding projection morphism $ \pi _h
  ^\prime \colon W ^\prime \rightarrow W _h $ with $ \pi _h ^\prime
  \circ i _h ^\prime = \mathrm{id} _{ W _h } $.  Then there also
  exists a projection morphism $ \pi _h \colon W \rightarrow W _h $
  such that $ \pi _h \circ i _h = \mathrm{id} _{ W _h } $.
\end{theorem}

\begin{proof}
  Take $ \pi _h = \pi _h ^\prime \circ f $.  Then $ \pi _h \circ i _h
  = \pi _h ^\prime \circ f \circ i _h = \pi _h ^\prime \circ i _h
  ^\prime = \mathrm{id} _{ W _h } $.
\end{proof}

\subsection{Modified inner product and Hodge decomposition}
As noted in the previous section, this generalized framework
introduces some new complications, due to the possible non-unitarity
of $ i _h $.  The following result shows that the subspace $ i _h W _h
\subset W $ can be identified with $ W _h $, endowed with a {\em
  modified} inner product $ \left\langle J _h \cdot, \cdot
\right\rangle _h $ instead of $ \left\langle \cdot , \cdot
\right\rangle _h $.  This defines a modified Hilbert complex,
which will be denoted by $ \left( i _h ^\ast W , \mathrm{d} _h \right)
$.

\begin{theorem}
  \label{thm:jacobianDeterminant}
  Let $ i _h \colon W _h \hookrightarrow W $ be a morphism of Hilbert
  complexes, and define $ J _h ^k = i _h ^{ k \ast } i _h ^k \colon W
  _h ^k \rightarrow W _h ^k $ for each $k$.  Then
  \begin{equation*}
    \left\langle J _h u _h , v _h \right\rangle _h =    \left\langle i
      _h u _h , i _h v _h \right\rangle , \quad \forall u _h , v _h
    \in W _h ^k ,
  \end{equation*}
  is an inner product, which defines a Hilbert space structure on $ W
  _h ^k $.
\end{theorem}

\begin{proof}
  $ \left\langle i _h u _h , i _h v _h \right\rangle = \left\langle i
    _h ^\ast i _h u _h , v _h \right\rangle _h = \left\langle J _h u
    _h , v _h \right\rangle _h $.  This is an inner product, since $ i
  _h $ is linear and injective.  Moreover, $ W _h ^k $ is closed with
  respect to the induced norm, since $ \left\lVert i _h v _h
  \right\rVert \leq \left\lVert i _h \right\rVert \left\lVert v _h
  \right\rVert $ and $ i _h $ is bounded, so this is indeed a Hilbert
  space.
\end{proof}

\begin{remark}
  We use the notation $ J _h $ due to the similarity with the Jacobian
  determinant used in the ``change of variables'' formula for
  integration.  Note that, although each $ J _h ^k \colon W _h ^k
  \rightarrow W _h ^k $ is a bounded linear map, $ J _h $ is not
  necessarily a Hilbert complex automorphism.  This is because, in
  general, $ \mathrm{d} $ commutes with $ i _h $ but not with its
  adjoint $ i _h ^\ast $.  Also, clearly $ i _h ^k $ is unitary if and
  only if $ J _h ^k = \mathrm{id} _{ W _h ^k } $.
\end{remark}

Now, if $ i _h $ does not preserve the inner product, in particular it
does not preserve orthogonality: that is, $ u _h \perp v _h $ does not
imply $ i _h u _h \perp i _h v _h $.  This has significant
implications for the Hodge decomposition, since although $ W _h ^k =
\overline{\mathfrak{B} _h ^k} \oplus \mathfrak{H} _h ^k \oplus
\mathfrak{Z} _h ^{\smash{k \perp _{W_h}}} $ is $ W _h $-orthogonal, it
is generally not $ i _h ^\ast W $-orthogonal.  Therefore, we define
the new, modified subspaces
\begin{equation*}
  \mathfrak{H} _h ^{\prime k} = \left\{ z \in \mathfrak{Z}  _h
    ^k \;\middle\vert\; i _h z \perp i _h \mathfrak{B}  _h ^k
  \right\}, \qquad \mathfrak{Z} _h ^{\smash{k \perp \prime _W
    }} = \left\{ v \in W _h ^k \;\middle\vert\; i _h v \perp i _h
    \mathfrak{Z}  _h ^k   \right\} .
\end{equation*}
This gives a modified Hodge decomposition $ W _h ^k = \overline{
  \mathfrak{B} _h ^k} \oplus \mathfrak{H} _h ^{\prime k} \oplus
\mathfrak{Z} _h ^{\smash{ k \perp\prime_W }} $, which is no
longer necessarily $ W _h $-orthogonal, but is now $ i _h ^\ast W
$-orthogonal.  As before, this also gives a modified Hodge
decomposition for the domain complex $ V _h ^k = \overline{
  \mathfrak{B} _h ^k} \oplus \mathfrak{H} _h ^{\prime k} \oplus
\mathfrak{Z} _h ^{ k \perp \prime } $.

\subsection{Stability and convergence of the mixed method}

Let $ \left( W, \mathrm{d} \right) $ be a closed Hilbert complex with
domain complex $ \left( V, \mathrm{d} \right) $.  To approximate a
solution to the mixed variational problem \eqref{eqn:mixedProblem},
suppose that $ \left( W _h , \mathrm{d} _h \right) $ is another
Hilbert complex with domain complex $ \left( V _h , \mathrm{d} _h
\right) $, and that we have morphisms $ i _h \colon V _h
\hookrightarrow V $ and $ \pi _h \colon V \rightarrow V _h $ such that
$ \pi _h ^k \circ i _h ^k = \mathrm{id} _{ V _h ^k } $ for each $k$.
We assume that $ i _h $ is $W$-bounded, so that it also can be
extended to $ W _h \hookrightarrow W $, but that $ \pi _h $ might only
be $V$-bounded.  Then consider the solution of the following mixed
variational problem: Find $ \left( \sigma _h , u _h , p _h \right) \in
V _h ^{ k - 1 } \times V _h ^k \times \mathfrak{H} _h ^k $ satisfying
\begin{equation}
  \label{eqn:discreteProblem}
  \begin{alignedat}{2}
    \left\langle \sigma _h , \tau _h \right\rangle _h - \left\langle u
      _h , \mathrm{d} _h \tau _h \right\rangle _h &= 0, &\quad \forall
    \tau _h &\in V _h ^{ k - 1 } ,\\
    \left\langle \mathrm{d} _h \sigma _h , v _h \right\rangle _h +
    \left\langle \mathrm{d} _h u _h , \mathrm{d} _h v _h \right\rangle
    _h + \left\langle p _h , v _h \right\rangle _h &= \left\langle f
      _h , v _h \right\rangle _h , &\quad \forall v _h
    &\in V _h ^k , \\
    \left\langle u _h , q _h \right\rangle _h &= 0 , &\quad \forall q
    _h &\in \mathfrak{H} _h ^k .
  \end{alignedat}
\end{equation}
This corresponds to the generalized variational problem
\eqref{eqn:generalizedGalerkin} with bilinear form
\begin{multline*}
  B _h \left( \sigma _h , u _h , p _h ; \tau _h , v _h , q _h \right)
  = \left\langle \sigma _h , \tau _h \right\rangle _h - \left\langle u
    _h , \mathrm{d}_h \tau _h \right\rangle _h  \\
  + \left\langle \mathrm{d}_h \sigma _h , v _h \right\rangle _h +
  \left\langle \mathrm{d}_h u _h , \mathrm{d}_h v _h \right\rangle _h +
  \left\langle p _h , v _h \right\rangle _h - \left\langle u _h , q _h
  \right\rangle _h 
\end{multline*}
and functional $ F _h \left( \tau _h , v _h , q _h \right) =
\left\langle f _h , v _h \right\rangle _h $.  The following theorem
establishes the inf-sup condition for the mixed method
\eqref{eqn:discreteProblem}.

\begin{theorem}
  \label{thm:discreteStability}
  Let $ \left( V , \mathrm{d} \right) $ be the domain complex of a
  closed Hilbert complex $ \left( W , \mathrm{d} \right) $, and let $
  \left( V _h , \mathrm{d} _h \right) $ be a family of domain
  complexes of closed Hilbert complexes $ \left( W _h , \mathrm{d} _h
  \right) $, equipped with uniformly $W$-bounded inclusion morphisms $
  i _h \colon V _h \hookrightarrow V $ and $V$-bounded projection
  morphisms $ \pi _h \colon V \rightarrow V _h $ satisfying $ \pi _h
  ^k \circ i _h ^k = \mathrm{id} _{ V _h ^k } $.  Then there exists a
  constant $ \gamma _h > 0 $, depending only on $ c _P $ and the norms
  of $ i _h $ and $ \pi _h $, such that for any $ \left( \sigma _h , u
    _h , p _h \right) \in V _h ^{ k - 1 } \times V _h ^k \times
  \mathfrak{H} _h ^k $, there exists $ \left( \tau _h , v _h , q _h
  \right) \in V _h ^{ k - 1 } \times V _h ^k \times \mathfrak{H} _h ^k
  $ where
  \begin{multline*}
    B _h \left( \sigma _h , u _h , p _h ; \tau _h , v _h , q _h
    \right) \\
    \geq \gamma _h \left( \left\lVert \sigma _h \right\rVert _{ V _h }
      + \left\lVert u _h \right\rVert _{ V _h } + \left\lVert p _h
      \right\rVert _h \right) \left( \left\lVert \tau _h \right\rVert
      _{ V _h } + \left\lVert v _h \right\rVert _{ V _h } +
      \left\lVert q _h \right\rVert _h \right) .
  \end{multline*}
\end{theorem}

\begin{proof}
  This is just \autoref{thm:inf-sup} applied to the Hilbert complex $
  \left( V _h , \mathrm{d} _h \right) $, combined with the fact that
  the Poincar\'e constant is $ c _P \left\lVert \pi _h \right\rVert
  \left\lVert i _h \right\rVert $ by \autoref{thm:discretePoincare}.
\end{proof}

\begin{remark}
  Since we have assumed that the morphisms $ i _h $ and $ \pi _h $ are
  uniformly bounded with respect to $h$, it follows that the inf-sup
  constants $ \gamma _h $ can be bounded below by some constant, which
  is independent of $h$.
\end{remark}

The goal, for the remainder of this section, will be to control the
error
\begin{equation*}
  \left\lVert \sigma - i _h \sigma _h \right\rVert _V + \left\lVert
    u - i _h u _h \right\rVert _V + \left\lVert p - i _h p _h
  \right\rVert ,
\end{equation*}
where $ \left( \sigma , u , p \right) $ is a solution to
\eqref{eqn:mixedProblem} and $ \left( \sigma _h , u _h , p _h \right)
$ is a solution to \eqref{eqn:discreteProblem}.  To do this, it will
be helpful to introduce the following modified mixed problem on $ i _h
^\ast V $: Find $ \left( \sigma _h ^\prime , u _h ^\prime , p _h
  ^\prime \right) \in V _h ^{ k - 1 } \times V _h ^k \times
\mathfrak{H} _h ^{ \prime k } $ satisfying
\begin{equation}
  \label{eqn:modifiedProblem}
  \begin{alignedat}{2}
    \left\langle J _h \sigma _h ^\prime , \tau _h \right\rangle _h -
    \left\langle J _h u _h ^\prime , \mathrm{d} _h \tau _h
    \right\rangle _h &= 0, &\quad \forall
    \tau _h &\in V _h ^{ k - 1 } ,\\
    \left\langle J _h \mathrm{d} _h \sigma _h ^\prime , v _h
    \right\rangle _h + \left\langle J _h \mathrm{d} _h u _h ^\prime ,
      \mathrm{d} _h v _h \right\rangle _h + \left\langle J _h p
      ^\prime _h , v _h \right\rangle _h &= \left\langle i _h ^\ast f
      , v _h \right\rangle _h , &\quad \forall v _h
    &\in V _h ^k , \\
    \left\langle J _h u _h ^\prime , q _h ^\prime \right\rangle _h &=
    0 , &\quad \forall q _h ^\prime &\in \mathfrak{H} _h ^{\prime k }
    .
  \end{alignedat}
\end{equation}
This has the corresponding bilinear form
\begin{multline*}
  B _h ^\prime \left( \sigma _h ^\prime , u _h ^\prime , p _h ^\prime ;
    \tau _h , v _h , q _h ^\prime \right) = \left\langle J _h \sigma _h
    ^\prime , \tau _h \right\rangle _h - \left\langle J _h
    u _h ^\prime , \mathrm{d}_h \tau _h \right\rangle _h \\
  + \left\langle J _h \mathrm{d}_h \sigma _h ^\prime , v _h
  \right\rangle _h + \left\langle J _h \mathrm{d}_h u _h ^\prime ,
    \mathrm{d}_h v _h \right\rangle _h + \left\langle J _h p _h ^\prime
    , v _h \right\rangle _h - \left\langle J _h u _h ^\prime ,
    q _h ^\prime \right\rangle _h ,
\end{multline*}
and the functional $ F _h ^\prime \left( \tau _h , v _h , q _h ^\prime
\right) = \left\langle i _h ^\ast f , v _h \right\rangle _h $.

This is precisely equivalent to the mixed problem on the subcomplex $
i _h V _h \subset V $, which has the bounded cochain projection $ i _h
\circ \pi _h \colon V \rightarrow i _h V _h $.  Therefore, the
stability and convergence analysis of \citet{ArFaWi2010} can be
applied immediately to this modified discrete problem.  In the end, we
will obtain the desired bound by applying the triangle inequality,
\begin{multline}
  \label{eqn:triangleInequality}
  \left\lVert \sigma - i _h \sigma _h \right\rVert _V + \left\lVert u
    - i _h u _h \right\rVert _V + \left\lVert p - i _h p _h
  \right\rVert\ \\
  \leq \left\lVert \sigma - i _h \sigma _h ^\prime \right\rVert _V +
  \left\lVert u - i _h u _h ^\prime \right\rVert _V +
  \left\lVert p - i _h p _h ^\prime \right\rVert \\
  + \left\lVert i _h \left( \sigma _h - \sigma_h ^\prime \right)
  \right\rVert _V + \left\lVert i _h \left( u _h - u_h ^\prime \right)
  \right\rVert _V + \left\lVert i _h \left( p _h - p_h ^\prime \right)
  \right\rVert .
\end{multline}
Observe that, since $ i _h $ is bounded, we can write
\begin{multline*}
  \left\lVert i _h \left( \sigma _h - \sigma_h ^\prime \right)
  \right\rVert _V + \left\lVert i _h \left( u _h - u_h ^\prime \right)
  \right\rVert _V + \left\lVert i _h \left( p _h - p_h ^\prime \right)
  \right\rVert \\
  \leq C \left( \left\lVert \sigma _h - \sigma_h ^\prime \right\rVert
    _{V _h} + \left\lVert u _h - u_h ^\prime \right\rVert _{V _h} +
    \left\lVert p _h - p_h ^\prime \right\rVert _h \right) ,
\end{multline*}
so it will suffice to control the error between solutions to
\eqref{eqn:discreteProblem} and \eqref{eqn:modifiedProblem} in $ V _h
$.

\begin{theorem}
  \label{thm:variationalPerturbation}
  Under the assumptions of \autoref{thm:discreteStability}, suppose
  that $ \left( \sigma _h , u _h , p _h \right) \in V _h ^{ k - 1 }
  \times V _h ^k \times \mathfrak{H} _h ^k $ is a solution to
  \eqref{eqn:discreteProblem} and $ \left( \sigma _h ^\prime , u _h
    ^\prime , p _h ^\prime \right) \in V _h ^{ k - 1 } \times V _h ^k
  \times \mathfrak{H} _h ^{ \prime k } $ is a solution to
  \eqref{eqn:modifiedProblem}.  Then
  \begin{equation*}
    \left\lVert \sigma _h - \sigma _h ^\prime \right\rVert _{ V _h } +
    \left\lVert u _h - u _h ^\prime \right\rVert _{ V _h } +
    \left\lVert p _h - p _h ^\prime \right\rVert _h \leq C \left(
      \left\lVert f _h - i _h ^\ast f \right\rVert _h + \left\lVert I
        - J _h \right\rVert \left\lVert f \right\rVert \right) .
  \end{equation*}
\end{theorem}

\begin{proof}
  For any $ \left( \tau , v, q \right) \in V _h ^{ k - 1 } \times V _h
  ^k \times \mathfrak{H} _h ^k $, we can write
  \begin{multline*}
    B _h \left( \sigma _h - \tau , u _h - v , p _h - q ; \tau _h , v
      _h , q _h \right) = B _h \left( \sigma _h - \sigma _h ^\prime ,
      u _h - u _h ^\prime , p _h - p _h ^\prime ; \tau _h , v _h , q
      _h \right) \\
    + B _h \left( \sigma _h ^\prime - \tau , u _h ^\prime - v , p _h
      ^\prime - q ; \tau _h , v _h , q _h \right) .
  \end{multline*}
  Ignoring the first term momentarily, observe for the second term
  that
  \begin{multline*} 
    B _h \left( \sigma _h ^\prime , u _h ^\prime , p _h ^\prime ; \tau
      _h , v _h , q _h \right) = B _h ^\prime \left( \sigma _h ^\prime
      , u _h ^\prime , p _h ^\prime ; \tau _h , v _h , q _h \right) \\
    + \left\langle \left( I - J _h \right) \sigma _h ^\prime , \tau _h
    \right\rangle _h - \left\langle \left( I - J _h \right) u _h
      ^\prime , \mathrm{d} _h \tau _h \right\rangle _h + \left\langle
      \left( I - J _h \right) \mathrm{d} _h \sigma _h ^\prime , v _h
    \right\rangle _h \\+ \left\langle \left( I - J _h \right)
      \mathrm{d} _h u _h ^\prime , \mathrm{d} _h v _h \right\rangle _h
    + \left\langle \left( I - J _h \right) p ^\prime _h , v _h
    \right\rangle _h - \left\langle \left( I - J _h \right) u _h
      ^\prime , q _h \right\rangle _h ,
  \end{multline*} 
  so by the variational principles \eqref{eqn:discreteProblem} and
  \eqref{eqn:modifiedProblem},
  \begin{align*}
    B _h ^\prime \left( \sigma _h ^\prime , u _h ^\prime , p _h
      ^\prime ; \tau _h , v _h , q _h \right) &= \left\langle i _h
      ^\ast f , v _h \right\rangle _h - \left\langle J _h u _h ^\prime
      , q _h \right\rangle _h ,\\
    B _h \left( \sigma _h , u _h , p _h ; \tau _h , v _h , q _h
    \right) &= \left\langle f _h , v _h \right\rangle _h .
  \end{align*}
  Therefore,
  \begin{multline*}
    B _h \left( \sigma _h - \sigma _h ^\prime , u _h - u _h ^\prime ,
      p _h - p _h ^\prime ; \tau _h , v _h , q _h \right) =
    \left\langle f _h - i _h ^\ast f , v _h \right\rangle _h +
    \left\langle u _h ^\prime , q _h \right\rangle _h \\
    - \left\langle \left( I - J _h \right) \sigma _h ^\prime , \tau _h
    \right\rangle _h + \left\langle \left( I - J _h \right) u _h
      ^\prime , \mathrm{d} _h \tau _h \right\rangle _h - \left\langle
      \left( I - J _h \right) \mathrm{d} _h \sigma _h ^\prime , v _h
    \right\rangle _h \\- \left\langle \left( I - J _h \right)
      \mathrm{d} _h u _h ^\prime , \mathrm{d} _h v _h \right\rangle _h
    - \left\langle \left( I - J _h \right) p ^\prime _h , v _h
    \right\rangle _h ,
  \end{multline*}
  so using the boundedness of the bilinear form and Cauchy--Schwarz,
  we get the upper bound
  \begin{multline*}
    B _h \left( \sigma _h - \tau , u _h - v , p _h - q ; \tau _h , v
      _h , q _h \right) \\
    \leq C \Bigl( \left\lVert f _h - i _h ^\ast f \right\rVert _h +
    \left\lVert P _{ \mathfrak{H} _h } u _h ^\prime \right\rVert _h +
    \left\lVert I - J _h \right\rVert \left( \left\lVert \sigma _h
        ^\prime \right\rVert _{ V _h } + \left\lVert u _h ^\prime
      \right\rVert _{ V _h } + \left\lVert p _h ^\prime \right\rVert
      _h \right) \\
    + \left\lVert \sigma _h ^\prime - \tau \right\rVert _{ V _h } +
    \left\lVert u _h ^\prime - v \right\rVert _{ V _h } + \left\lVert
      p _h ^\prime - q \right\rVert _h \Bigr) \Bigl( \left\lVert \tau
      _h \right\rVert _{ V _h } + \left\lVert v _h \right\rVert _{ V
      _h } + \left\lVert q _h \right\rVert _h \Bigr) .
  \end{multline*}
  Next, \autoref{thm:discreteStability} gives the lower bound
  \begin{multline*}
    B _h \left( \sigma _h - \tau , u _h - v , p _h - q ; \tau _h , v
      _h , q _h \right) \\
    \geq \gamma _h \left( \left\lVert \sigma _h - \tau \right\rVert _{
        V _h } + \left\lVert u _h - v \right\rVert _{ V _h } +
      \left\lVert p _h - q \right\rVert _h \right) \left( \left\lVert
        \tau _h \right\rVert _{ V _h } + \left\lVert v _h \right\rVert
      _{ V _h } + \left\lVert q _h \right\rVert _h \right) 
  \end{multline*}
  for some $ \left( \tau _h , v _h , q _h \right) \in V _h ^{ k - 1 }
  \times V _h ^k \times \mathfrak{H} _h ^k $, where $ \gamma _h $ can
  be bounded below independently of $h$.  Therefore, combining the
  upper and lower bounds and dividing out $ \left\lVert \tau _h
  \right\rVert _{ V _h } + \left\lVert v _h \right\rVert _{ V _h } +
  \left\lVert q _h \right\rVert _h $, we get
  \begin{multline*}
    \left\lVert \sigma _h - \tau \right\rVert _{ V _h } + \left\lVert
      u _h - v \right\rVert _{ V _h } + \left\lVert p _h - q
    \right\rVert _h \\
    \leq C \Bigl( \left\lVert f _h - i _h ^\ast f \right\rVert _h +
    \left\lVert P _{ \mathfrak{H} _h } u _h ^\prime \right\rVert _h +
    \left\lVert I - J _h \right\rVert \left( \left\lVert \sigma _h
        ^\prime \right\rVert _{ V _h } + \left\lVert u _h ^\prime
      \right\rVert _{ V _h } + \left\lVert p _h ^\prime \right\rVert
      _h \right) \\
    + \left\lVert \sigma _h ^\prime - \tau \right\rVert _{ V _h } +
    \left\lVert u _h ^\prime - v \right\rVert _{ V _h } + \left\lVert
      p _h ^\prime - q \right\rVert _h \Bigr) .
  \end{multline*}
  This expression can be simplified considerably by choosing $ \tau =
  \sigma _h ^\prime $, $ v = u _h ^\prime $, and $ q = P _{
    \mathfrak{H} _h } p _h ^\prime $, so applying the triangle
  inequality gives the error estimate
  \begin{multline*}
    \left\lVert \sigma _h - \sigma _h ^\prime \right\rVert _{ V _h } +
    \left\lVert u _h - u _h ^\prime \right\rVert _{ V _h } +
    \left\lVert p _h - p _h ^\prime \right\rVert _h  \\
    \leq C \left( \left\lVert f _h - i _h ^\ast f \right\rVert _h +
      \left\lVert P _{ \mathfrak{H} _h } u _h ^\prime \right\rVert _h
      + \left\lVert I - J _h \right\rVert \left\lVert f \right\rVert +
      \left\lVert p _h ^\prime - q \right\rVert _h \right) .
  \end{multline*}
  All that remains is to deal with the terms $ \left\lVert P _{
      \mathfrak{H} _h } u _h ^\prime \right\rVert _h $ and $
  \left\lVert p _h ^\prime - q \right\rVert _h $.  First, since $ u _h
  ^\prime $ is $ i _h ^\ast V $-orthogonal to $ \mathfrak{H} _h ^{
    \prime k } $, the modified Hodge decomposition lets us write $ u
  ^\prime _h = u ^\prime _{ \mathfrak{B} } + u ^\prime _\perp $, where
  $ u ^\prime _{ \mathfrak{B} } \in \mathfrak{B} _h ^k $ and $ u
  ^\prime _\perp \in \mathfrak{Z} _h ^{ k \perp \prime } $.  Now,
  observe that $ P _{ \mathfrak{H} _h } u _{ \mathfrak{B} } ^\prime =
  0 $ since $ \mathfrak{B} _h ^k \perp \mathfrak{H} _h ^k $, and
  furthermore $ P _{ \mathfrak{H} _h } J _h u _\perp ^\prime = 0 $
  since $ u _\perp ^\prime \in \mathfrak{Z} _h ^{ k \perp \prime } $
  implies $ J _h u _\perp ^\prime \perp \mathfrak{Z} _h ^k $.
  Therefore,
  \begin{equation*}
    \left\lVert P _{ \mathfrak{H}  _h } u _h ^\prime \right\rVert _h =
    \left\lVert P _{ \mathfrak{H}  _h } u _\perp ^\prime \right\rVert
    _h = \left\lVert P _{ \mathfrak{H}  _h } \left( I - J _h \right) u
      _\perp ^\prime \right\rVert _h \leq C \left\lVert I - J _h
    \right\rVert \left\lVert f \right\rVert .
  \end{equation*}
  Next, since $ p _h ^\prime \in \mathfrak{H} _h ^{ \prime k } \subset
  \mathfrak{Z} _h ^k $, the Hodge decomposition gives $ p _h ^\prime =
  P _{ \mathfrak{B} _h } p _h ^\prime + P _{ \mathfrak{H} _h } p _h
  ^\prime = P _{ \mathfrak{B} _h } p _h ^\prime + q $.  Also, similar
  to the previous term, since $ p _h ^\prime \in \mathfrak{H} _h ^{
    \prime k } $ we have $ J _h p _h ^\prime \perp \mathfrak{B} _h ^k
  $, so $ P _{ \mathfrak{B} _h } J _h p _h ^\prime = 0 $.  Thus,
  \begin{equation*}
    \left\lVert p _h ^\prime - q \right\rVert _h = \left\lVert P _{
        \mathfrak{B}  _h } p _h ^\prime \right\rVert _h = \left\lVert
      P _{ \mathfrak{B}  _h } \left( I - J _h \right) p _h ^\prime
    \right\rVert _h \leq C \left\lVert I - J _h \right\rVert
    \left\lVert f \right\rVert .
  \end{equation*}
  Therefore, these two terms can be combined with the existing $
  \left\lVert I - J _h \right\rVert \left\lVert f \right\rVert $ term,
  leaving the final error estimate,
  \begin{equation*}
    \left\lVert \sigma _h - \sigma _h ^\prime \right\rVert _{ V _h } +
    \left\lVert u _h - u _h ^\prime \right\rVert _{ V _h } +
    \left\lVert p _h - p _h ^\prime \right\rVert _h 
    \leq C \left( \left\lVert f _h - i _h ^\ast f \right\rVert _h 
      + \left\lVert I - J _h \right\rVert \left\lVert f \right\rVert
    \right) ,
  \end{equation*}
  as desired, which completes the proof.
\end{proof}

\begin{corollary}
  If $ \left( \sigma , u , p \right) \in V ^{ k - 1 } \times V ^k
  \times \mathfrak{H} ^k $ is a solution to \eqref{eqn:mixedProblem}
  and $ \left( \sigma _h , u _h , p _h \right) \in V _h ^{ k - 1 }
  \times V _h ^k \times \mathfrak{H} _h ^k $ is a solution to
  \eqref{eqn:discreteProblem}, then
    \begin{multline*}
      \left\lVert \sigma - i _h \sigma _h \right\rVert _V +
      \left\lVert u - i _h u _h \right\rVert _V + \left\lVert p - i _h
        p _h \right\rVert \\
      \leq C \bigl( \inf _{ \tau \in i _h V _h ^{ k - 1 } }
      \left\lVert \sigma - \tau \right\rVert _V + \inf _{ v \in i _h V
        _h ^k } \left\lVert u - v \right\rVert _V + \inf _{ q \in i _h
        V _h ^k } \left\lVert p - q \right\rVert _V + \mu \inf _{ v
        \in i _h V _h
        ^k } \left\lVert P _{ \mathfrak{B} } u - v \right\rVert _V \\
      + \left\lVert f _h - i _h ^\ast f \right\rVert _h + \left\lVert
        I - J _h \right\rVert \left\lVert f \right\rVert \bigr) ,
  \end{multline*}
  where $ \mu = \mu _h ^k = \displaystyle\sup _{ \substack{r \in
      \mathfrak{H} ^k \\ \left\lVert r \right\rVert = 1 } }
  \left\lVert \left( I - i _h ^k \pi _h ^k \right) r \right\rVert $.
\end{corollary}

\begin{proof}
  Use the triangle inequality, as in \eqref{eqn:triangleInequality},
  and then apply \autoref{thm:subcomplex} and
  \autoref{thm:variationalPerturbation} to bound the respective error
  terms.
\end{proof}

This theorem establishes convergence, as long as our approximations
satisfy $ \left\lVert I - J _h \right\rVert \rightarrow 0 $ and $
\left\lVert f _h - i _h ^\ast f \right\rVert _h \rightarrow 0 $ when $
h \rightarrow 0 $.  This raises the question of how to choose $ f _h
\in V _h ^k $; although clearly $ f _h = i _h ^\ast f $ will work, in
many cases this cannot be computed efficiently.  The next result
demonstrates that, if $ \Pi _h \colon W ^k \rightarrow W _h ^k $ is
any bounded linear projection (i.e., satisfying $ \Pi _h \circ i _h ^k
= \mathrm{id} _{ W _h ^k } $), then simply choosing $ f _h = \Pi _h f
$ is sufficient to get a quasi-optimally convergent solution.

\begin{theorem}
  If $ \Pi _h \colon W ^k \rightarrow W _h ^k $ is a family of linear
  projections, bounded uniformly with respect to $h$, then we have the
  inequality
  \begin{equation*}
    \left\lVert \Pi _h f - i _h ^\ast f \right\rVert _h \leq C \bigl(
    \left\lVert I - J _h \right\rVert \left\lVert f \right\rVert +
    \displaystyle\inf _{ \phi  \in i _h W _h ^k } \left\lVert f - \phi 
    \right\rVert \bigr) .
  \end{equation*} 
\end{theorem}

\begin{proof}
  Using the triangle inequality, we write
  \begin{align*}
    \left\lVert \left( \Pi _h - i _h ^\ast \right) f \right\rVert _h
    &\leq \left\lVert \left( \Pi _h - i _h ^\ast i _h \Pi _h \right) f
    \right\rVert _h + \left\lVert \left( i _h ^\ast - i _h ^\ast i _h
        \Pi _h \right) f \right\rVert _h \\
    &= \left\lVert \left( I - i _h ^\ast i _h \right) \Pi _h f
    \right\rVert _h + \left\lVert i _h ^\ast \left( I - i _h \Pi _h
      \right) f \right\rVert _h \\
    &\leq \left\lVert I - J _h \right\rVert \left\lVert \Pi _h f
    \right\rVert _h + \left\lVert i _h ^\ast \right\rVert \left\lVert
      \left( I - i _h \Pi _h \right) f \right\rVert \\
    &\leq C \bigl( \left\lVert I - J _h \right\rVert \left\lVert f
    \right\rVert + \inf _{ \phi \in i _h W _h ^k } \left\lVert f -
      \phi \right\rVert \bigr) ,
  \end{align*}
  where the final step follows from the $W$-boundedness of $ \Pi _h $ and
  the quasi-optimality property of $ I - i _h \Pi _h $, i.e., $ \left(
    I - i _h \Pi _h \right) f = \left( I - i _h \Pi _h \right) \left(
    f - \phi \right) $ for any $ \phi \in i _h W _h ^k $.
\end{proof}

\subsection{Remarks on obtaining improved error estimates}

\citet{ArFaWi2010} were also able to obtain improved error estimates
by making some additional assumptions: namely, that $ \pi _h $ is
$W$-bounded rather than merely $V$-bounded, and that the Hilbert
complex $V$ satisfies a certain compactness property.  With these
assumptions, the continuous solution operator $K \colon W ^k
\rightarrow W ^k $ becomes a compact operator, and hence converts the
pointwise convergence of $ I - \pi _h \rightarrow 0 $ (which follows
from the quasi-optimality property) to norm convergence.  This norm
convergence is essential for applying the so-called ``Aubin--Nitsche
trick'' (also known as ``$ L ^2 $ lifting''), where one obtains
improved estimates by applying the solution operator to the error term
itself.  Roughly speaking, one needs norm convergence, rather than
pointwise convergence, since the solution operator is being applied to
quantities that depend on the parameter $h$.

However, there are no such improved estimates for the additional error
terms obtained in the previous subsection.  Essentially, this is
because norm convergence is already required for $ \left\lVert I - J
  _h \right\rVert \rightarrow 0 $ as $ h \rightarrow 0 $, and there is
no analogous quasi-optimality result for $ J _h $ as there is for $
\pi _h $.  Therefore, these terms remain the same, and the improved
estimates only apply to the terms already analyzed by
\citet{ArFaWi2010} for the subcomplex case.

\subsection{Convergence of the eigenvalue problem}

While we have primarily focused on the numerical approximation of the
mixed variational problem, \citet{ArFaWi2010} also analyzed an
eigenvalue problem associated to the Hodge Laplacian.  The extension
of their eigenvalue convergence result to non-subcomplexes is fairly
straightforward, and follows from the results already given in this
section, as we will now show.

Consider the eigenvalue problem
\begin{equation}
  \label{eqn:evProblem}
  \begin{alignedat}{2}
    \left\langle \sigma , \tau \right\rangle - \left\langle u ,
      \mathrm{d} \tau \right\rangle &= 0, &\quad \forall \tau &\in V
    ^{ k - 1 } ,\\
    \left\langle \mathrm{d} \sigma, v \right\rangle + \left\langle
      \mathrm{d} u, \mathrm{d} v \right\rangle + \left\langle p, v
    \right\rangle &= \lambda \left\langle u, v \right\rangle , &\quad
    \forall
    v &\in V ^k ,\\
    \left\langle u, q \right\rangle &= 0 , &\quad \forall q &\in
    \mathfrak{H} ^k ,
  \end{alignedat}
\end{equation}
the discrete problem
\begin{equation}
  \label{eqn:evDiscreteProblem}
  \begin{alignedat}{2}
    \left\langle \sigma _h , \tau _h \right\rangle _h - \left\langle u
      _h , \mathrm{d} _h \tau _h \right\rangle _h &= 0, &\quad \forall
    \tau _h &\in V _h ^{ k - 1 } ,\\
    \left\langle \mathrm{d} _h \sigma _h , v _h \right\rangle _h +
    \left\langle \mathrm{d} _h u _h , \mathrm{d} _h v _h \right\rangle
    _h + \left\langle p _h , v _h \right\rangle _h &= \lambda _h
    \left\langle u _h , v _h \right\rangle _h , &\quad \forall v _h
    &\in V _h ^k , \\
    \left\langle u _h , q _h \right\rangle _h &= 0 , &\quad \forall q
    _h &\in \mathfrak{H} _h ^k ,
  \end{alignedat}
\end{equation}
and the modified discrete problem
\begin{equation}
  \label{eqn:evModifiedProblem}
  \begin{alignedat}{2}
    \left\langle J _h \sigma _h ^\prime , \tau _h \right\rangle _h -
    \left\langle J _h u _h ^\prime , \mathrm{d} _h \tau _h
    \right\rangle _h &= 0, &\ \forall
    \tau _h &\in V _h ^{ k - 1 } ,\\
    \left\langle J _h \mathrm{d} _h \sigma _h ^\prime , v _h
    \right\rangle _h + \left\langle J _h \mathrm{d} _h u _h ^\prime ,
      \mathrm{d} _h v _h \right\rangle _h + \left\langle J _h p
      ^\prime _h , v _h \right\rangle _h &= \lambda _h ^\prime
    \left\langle u _h ^\prime , v _h \right\rangle _h , &\ \forall
    v _h &\in V _h ^k , \\
    \left\langle J _h u _h ^\prime , q _h ^\prime \right\rangle _h &=
    0 , &\ \forall q _h ^\prime &\in \mathfrak{H} _h ^{\prime k } .
  \end{alignedat}
\end{equation}
As shown by \citet[Theorem 3.19]{ArFaWi2010}, solutions to the
subcomplex problem \eqref{eqn:evModifiedProblem} converge to those of
\eqref{eqn:evProblem}, which follows immediately from the fact that $
i _h K _h ^\prime P _h $ converges to $K$ in the $ \mathcal{L} \left(
  W ^k , W ^k \right) $ operator norm.  We now show that this result
also holds for the problem \eqref{eqn:evDiscreteProblem}.

\begin{theorem}
  Let $ \left( V , \mathrm{d} \right) $ be the domain complex of a
  closed Hilbert complex $ \left( W , \mathrm{d} \right) $ satisfying
  the compactness property, and let $ \left( V _h , \mathrm{d} _h
  \right) $ be a family of domain complexes of closed Hilbert
  complexes $ \left( W _h , \mathrm{d} _h \right) $, equipped with
  morphisms $ i _h \colon W _h \hookrightarrow W $ and $ \pi _h \colon
  W \rightarrow W _h $ such that $ \pi _h ^k \circ i _h ^k =
  \mathrm{id} _{ W _h ^k } $, where $ i _h $ and $ \pi _h $ are
  bounded uniformly with respect to $h$.  Then the discrete eigenvalue
  problems \eqref{eqn:evDiscreteProblem} converge to the problem
  \eqref{eqn:evProblem}.
\end{theorem}

\begin{proof}
  It suffices to show that $ i _h K _h P _h $ converges to $K$ in the
  $ \mathcal{L} \left( W ^k, W ^k \right) $ operator norm.  (As stated
  by \citet{ArFaWi2010}, the sufficiency of norm convergence follows
  from \citet{BoBrGa2000}.)  Using the triangle inequality, we write
  \begin{equation*}
    \left\lVert K - i _h K _h P _h \right\rVert \leq \left\lVert K - i
      _h K _h ^\prime P _h \right\rVert + \left\lVert i _h \left( K _h
        ^\prime - K _h \right) P _h \right\rVert .
  \end{equation*}
  The first term on the right-hand side converges to zero, by
  \citet[Corollary 3.17]{ArFaWi2010}.  For the second term, recall
  that $ i _h $ and $ \pi _h $ are assumed to be bounded uniformly
  with respect to $h$, and since $ \left\lVert P _h \right\rVert =
  \left\lVert \pi _h P _{ i _h W _h } \right\rVert \leq \left\lVert
    \pi _h \right\rVert $, it follows that $ P _h $ is bounded
  uniformly with respect to $h$, as well.  Therefore, it suffices to
  control $ \left\lVert K _h ^\prime - K _h \right\rVert $ in $
  \mathcal{L} \left( W _h ^k , W _h ^k \right) $.  However, the
  earlier analysis in \autoref{thm:variationalPerturbation} shows that
  $ \left\lVert K _h ^\prime - K _h \right\rVert \leq C \left\lVert I
    - J _h \right\rVert $, which completes the proof of convergence.
\end{proof}

\section{Application to differential forms on Riemannian manifolds}
\label{sec:diffForms}

In this section, we apply the framework developed in
\autoref{sec:crimes} to the Hodge--de~Rham complex of differential
forms on a compact oriented Riemannian manifold.  We will begin by
first recalling the basic definitions of the de~Rham complex of smooth
forms; its completion as a Hilbert complex, called the $ L ^2
$-de~Rham complex; and the corresponding domain complex, which
dovetails with the theory of Sobolev spaces.  Next, we discuss the
general problem of approximating the de~Rham complex on a manifold $ M
$ by a family of ``nearby'' manifolds $ M _h $, each equipped with an
orientation-preserving diffeomorphism $ \varphi _h \colon M _h
\rightarrow M $.  We subsequently establish the correspondence between
this setup and the generalized Hilbert complex approximation framework
of \autoref{sec:crimes}, obtaining estimates for the appropriate maps,
as needed.  We then specialize the discussion a bit further by
considering the case when $M$ is a submanifold of some larger manifold
$N$; in this case, the approximating submanifolds $ M _h \subset N $
can be taken to lie in a tubular neighborhood of $M$, and $ \varphi _h
\colon M _h \rightarrow M $ is obtained by projection along normals.

Finally, we then look at the specific case where $ N = \mathbb{R}^n $,
and where we wish to approximate a solution on some $m$-dimensional
Euclidean hypersurface $ M \subset \mathbb{R}^n $, $ n = m + 1 $, by
finite elements defined on a piecewise-linear mesh $ M _h \subset
\mathbb{R}^n $.  This is now the realm of of surface finite element
methods, as analyzed in \citet{Dziuk1988,DeDz2007,Demlow2009}.  We
subsequently show how our results of the previous sections recover the
analysis framework and \emph{a priori} estimates of
\citet{Dziuk1988,DeDz2007,Demlow2009}, extending their results from
scalar functions on 2- and 3-surfaces to general $k$-forms on
arbitrary dimensional hypersurfaces.  We also indicate how our results
generalize the \emph{a priori} estimates of
\citet{Dziuk1988,Demlow2009} from nodal finite element methods for the
Laplace--Beltrami operator to mixed finite element methods for the
Hodge Laplacian.

\subsection{A brief review of Hodge--de~Rham theory}

Given a smooth, $m$-dimensional manifold $M$, let $ \Omega ^k (M) $
denote the space of smooth $k$-forms on $M$ for $ k = 0, 1, \ldots, m
$, and let $ \mathrm{d} ^k \colon \Omega ^k (M) \rightarrow \Omega ^{
  k + 1 } (M) $ be the exterior derivative for $ k = 0 , 1, \ldots, m
- 1 $.  Then $ \left( \Omega (M) , \mathrm{d} \right) $ is a cochain
complex,
\begin{equation*}
  \xymatrix@=2em{
    0 \ar[r] & \Omega ^0 (M) \ar[r]^-{\mathrm{d}} & \Omega ^1 (M)
    \ar[r]^-{\mathrm{d}} & \cdots \ar[r]^-{\mathrm{d}} & \Omega ^m (M)
    \ar[r] & 0 .
  }
\end{equation*}
called the \emph{de~Rham complex} on $M$.

Suppose that, in addition, $M$ is oriented and compact, and has a
Riemannian metric $g$.  Then, we can define the \emph{$ L ^2 $-inner
  product} of any $ u, v \in \Omega ^k (M) $ to be
\begin{equation*}
  \left\langle u, v \right\rangle _{ L ^2 \Omega (M) } = \int _M u \wedge \star _g v = \int
  _M \left\langle \! \left\langle u, v \right\rangle \! \right\rangle
  _g \mu _g .
\end{equation*}
Here, $ \star _g \colon \Omega ^k (M) \rightarrow \Omega ^{ m - k }
(M) $ is the Hodge star operator associated to the metric, $
\left\langle \! \left\langle \cdot , \cdot \right\rangle \!
\right\rangle _g \colon \Omega ^k (M) \times \Omega ^k (M) \rightarrow
C ^\infty (M) $ is the pointwise inner product induced by the metric
and $ \mu _g $ is the Riemannian volume form.  (The Hodge star is
defined precisely so that $ u \wedge \star _g v = \left\langle \!
  \left\langle u, v \right\rangle \!  \right\rangle _g \mu _g $, and
it follows that $ \star _g $ is an isometry.)  The Hilbert space $ L
^2 \Omega ^k (M) $ is then defined, for each $k$, to be the completion
of $ \Omega ^k (M) $ with respect to the $ L ^2 $-inner product.

To show that this forms a Hilbert complex $ \left( L ^2 \Omega (M) ,
  \mathrm{d} \right) $, we must now define the \emph{weak exterior
  derivative} $ \mathrm{d} ^k $ on some dense domain of $ L ^2 \Omega
^k (M) $.  Given $ u \in L ^2 \Omega ^k (M) $, we say that $w \in L ^2
\Omega ^{ k + 1 } (M) $ is the weak exterior derivative of $u$, and
write $ \mathrm{d} u = w $, if
\begin{equation*}
  \left\langle u, \mathrm{d} ^\ast v \right\rangle _{ L ^2 \Omega (M)
  } = \left\langle w, v \right\rangle _{ L ^2 \Omega (M) } , \quad
  \forall v \in \Omega ^{k+1} _c (M) ,
\end{equation*}
where $ \Omega ^{k+1} _c (M) $ denotes the space of smooth
$(k+1)$-forms with compact support. Therefore, one defines the dense
domains $ H \Omega ^k (M) \subset L ^2 \Omega ^k (M) $, consisting of
elements in $ L ^2 \Omega ^k (M) $ that have a weak exterior
derivative in $ L ^2 \Omega ^{ k + 1 } (M) $.  Thus, we have
\begin{equation*}
  \xymatrix@=2em{
    0 \ar[r] & H \Omega ^0 (M) \ar[r]^-{\mathrm{d}} & H \Omega ^1 (M)
    \ar[r]^-{\mathrm{d}} & \cdots \ar[r]^-{\mathrm{d}} & H \Omega ^m (M)
    \ar[r] & 0 .
  }
\end{equation*}
where each $ H \Omega ^k (M) $ can be given the graph inner product
\begin{equation*}
  \left\langle u, v \right\rangle _{ H \Omega (M) } = \left\langle u,
    v \right\rangle _{ L ^2 \Omega (M) } + \left\langle \mathrm{d} u ,
    \mathrm{d} v \right\rangle _{ L ^2 \Omega (M) } .
\end{equation*}
(Note the similarity with the definition of the Sobolev spaces $ H ^1
$, $ H \left( \operatorname{curl} \right) $, and $ H \left(
  \operatorname{div} \right) $.)  Since each $ H \Omega ^k (M) $ is
complete, it follows that $ \mathrm{d} ^k $ is a closed operator;
therefore, $ \left( L ^2 \Omega (M) , \mathrm{d} \right) $ is indeed a
Hilbert complex, and $ \left( H \Omega (M) , \mathrm{d} \right) $ is
the corresponding domain complex.  Furthermore, it can be shown that $
\left( L ^2 \Omega (M) , \mathrm{d} \right) $ satisfies the
compactness condition, so these Hilbert complexes are in fact closed
and satisfy the conditions necessary for the improved error estimates.
(For more details on the construction of these complexes, see
\citet{ArFaWi2010}.)

\subsection{Diffeomorphic Riemannian manifolds}

Let $ \left( M , g \right) $ be an oriented, compact, $m$-dimensional
Riemannian manifold, and suppose $ \left( M _h , g _h \right) $ is a
family of oriented, compact Riemannian manifolds, parametrized by $h$
and equipped with orientation-preserving diffeomorphisms $ \varphi _h
\colon M _h \rightarrow M $.  Now, since the pullback $ \varphi _h
^\ast \colon \Omega (M) \rightarrow \Omega \left( M _ h \right) $ and
pushforward $ \varphi _{ h \ast } \colon \Omega \left( M _h \right)
\rightarrow \Omega (M) $ commute with the exterior derivative, they
give a cochain isomorphism between the smooth de~Rham complexes $
\Omega \left( M _h \right) $ and $ \Omega \left( M \right) $.

We now show that these maps are bounded, and hence can be extended to
Hilbert complex isomorphisms between $ L ^2 \Omega \left( M _h \right)
$ and $ L ^2 \Omega (M) $, following the results of
\citet{Stern2010a}.  Given any point $ x \in M _h $, let $ \left\{ e
  _1, \ldots, e _m \right\} $ be a positively-oriented, $ g _h
$-orthonormal basis of the tangent space $ T _x M _h $, and let $
\left\{ f _1, \ldots, f _m \right\} $ be a positively-oriented, $ g
$-orthonormal basis of $ T _{ \varphi _h (x) } M $.  Then, with
respect to these bases, the tangent map $ T _x \varphi _h \colon T _x
M _h \rightarrow T _{ \varphi _h (x) } M $ can be represented by an $
m \times m $ matrix $\Phi$.  Moreover, since $\varphi _h$ is a
diffeomorphism, the matrix $\Phi$ has $m$ strictly positive singular
values,
\begin{equation*}
  \alpha _1 (x) \geq \cdots \geq \alpha _m (x) > 0 .
\end{equation*}
These singular values are orthogonally invariant, so they are
independent of the choice of basis at each $ x $ and $ \varphi _h (x)
$.  Hence, they are an intrinsic property of the diffeomorphism, and
thus we refer to them as the \emph{singular values of $ \varphi _h $ at
  $x$}.

\begin{theorem}[\citet{Stern2010a}, Corollary 6]
  \label{thm:pushforward}
  Let $ \left( M _h , g _h \right) $ and $ \left( M, g \right) $ be
  oriented, $m$-dimensional Riemannian manifolds, and let $ \varphi _h
  \colon M _h \rightarrow M $ be an orientation-preserving
  diffeomorphism with singular values $ \alpha _1 (x) \geq \cdots \geq
  \alpha _n (x) > 0 $ at each $ x \in M _h $.  Given $ p , q \in
  \left[ 1, \infty \right] $ such that $ 1/p + 1/q = 1 $, and some $ k
  = 0 , \ldots, m $, suppose that the product $ \left( \alpha _1
    \cdots \alpha _{ m - k } \right) ^{1/p} \left( \alpha _{ m - k + 1
    } \cdots \alpha _m \right) ^{ -1/q } $ is bounded uniformly on $M
  _h$.  Then, for any $ \omega \in L ^p \Omega ^k \left( M _h \right)
  $,
  \begin{multline*}
    \bigl\lVert \left( \alpha _1 \cdots \alpha _k \right) ^{ 1/q }
    \left( \alpha _{ k + 1 } \cdots \alpha _m \right) ^{-1/p}
    \bigr\rVert _\infty ^{-1} \left\lVert \omega
    \right\rVert _p  \\
    \leq \left\lVert \varphi _{h\ast} \omega \right\rVert _p \leq
    \bigl\lVert \left( \alpha _1 \cdots \alpha _{ m - k } \right)
    ^{1/p} \left( \alpha _{ m - k + 1 } \cdots \alpha _m \right) ^{
      -1/q }\bigr\rVert _\infty \left\lVert \omega \right\rVert _p .
  \end{multline*}
\end{theorem}

\begin{proof}[Sketch of proof]
  At each point, a $k$-form is $k$-linear and totally antisymmetric.
  Therefore, the pullback is controlled pointwise by the product of
  the $k$ largest singular values of $\varphi _h$, while the
  pushforward is controlled by the product of the $k$ largest singular
  values of $ \varphi _h ^{-1} $ (i.e., the reciprocals of the $k$
  smallest singular values of $\varphi _h $).  Thus, we obtain
  pointwise inequalities
  \begin{equation*}
    \left\lvert \varphi _h ^\ast \eta \right\rvert \leq \alpha _1 \cdots
    \alpha _k \left( \left\lvert \eta \right\rvert \circ \varphi _h 
    \right) , \qquad 
    \left\lvert \varphi _{h \ast} \omega \right\rvert \leq \bigl[ \left(
      \alpha _{ m - k + 1 } \cdots \alpha _m \right) ^{-1} \left\lvert
      \omega \right\rvert \bigr] \circ \varphi _h  ^{-1} .
  \end{equation*}
  For the $ L ^p $ upper bound, we can apply the pushforward
  inequality to get a factor of $ \left( \alpha _{ m - k + 1 } \cdots
    \alpha _m \right) ^{-p} $ in the integrand.  Using the change of
  variables theorem introduces the Jacobian determinant $ \alpha _1
  \cdots \alpha _m $, so multiplying by this gives a factor of $
  \alpha _1 \cdots \alpha _{ m - k } \left( \alpha _{ m - k + 1 }
    \cdots \alpha _m \right) ^{-p+1} $.  We can then use H\"older's
  inequality to pull out the $ L ^\infty $-norm of this expression,
  and raising to the exponent $ 1/p $ gives 
  \begin{multline*} 
    \bigl\lVert \left( \alpha _1 \cdots \alpha _{ m - k } \right)
    ^{1/p} \left( \alpha _{ m - k + 1 } \cdots \alpha _m \right)
    ^{-1+1/p} \bigr\rVert _\infty \\
    = \bigl\lVert \left( \alpha _1 \cdots \alpha _{ m - k } \right)
    ^{1/p} \left( \alpha _{ m - k + 1 } \cdots \alpha _m \right)
    ^{-1/q} \bigr\rVert _\infty ,
  \end{multline*}
  as desired.  The lower bound follows in a similar fashion, starting
  with the identity $ \omega = \varphi ^\ast _h \varphi _{ h \ast }
  \omega $ and applying the pointwise pullback inequality.
\end{proof}

Since $ M $ and $ M _h $ are compact, the uniform boundedness
hypothesis of this theorem is clearly satisfied.  Therefore, taking $
p = q = 2 $, it follows that the diffeomorphism $ \varphi _h $ induces
Hilbert complex isomorphisms $ \varphi _{ h \ast } \colon L ^2 \Omega
\left( M _h \right) \rightarrow L ^2 \Omega (M) $ and $ \varphi _h
^\ast \colon L ^2 \Omega \left( M \right) \rightarrow L ^2 \Omega
\left( M _h \right) $.

Now, take $ W = L ^2 \Omega (M) $, and suppose that we have discrete
subcomplexes $ W _h \subset L ^2 \Omega \left( M _h \right) $ with
inclusion morphisms $ i _h ^\prime \colon W _h \hookrightarrow L ^2
\Omega \left( M _h \right) $, as well as projection morphisms $ \pi _h
^\prime \colon L ^2 \Omega \left( M _h \right) \rightarrow W _h $
bounded uniformly in $h$.  Following the approach of
\autoref{thm:projection}, we can pull these back to obtain the
injection morphisms $ i _h = \varphi _{ h \ast } \circ i _h ^\prime
\colon W _h \hookrightarrow W $ and projection morphisms $ \pi _h =
\pi _h ^\prime \circ \varphi _h ^\ast \colon W \rightarrow W _h $,
which satisfy $ \pi _h \circ i _h = \mathrm{id} _{ W _h } $.  An
important consequence of this is stated in the following corollary of
\autoref{thm:projection} and \autoref{thm:pushforward}.

\begin{corollary}
  \label{cor:inducedfeec}
  Orientation-preserving diffeomorphisms induce an equivalence of
  families of finite element subcomplexes of the $ L ^2 $-de~Rham
  complex with bounded cochain projections.  In particular, any
  triangulation $ \mathcal{T} _h \rightarrow M $ gives corresponding $
  \mathcal{P} _r ^- $ and $ \mathcal{P} _r $ families
  (cf.~\citet{ArFaWi2006,ArFaWi2010}) of piecewise-polynomial
  differential forms on $M$.
\end{corollary}

Finally, let us see how this definition of $ i _h $ can be used to
control the error term $ \left\lVert I - J _h \right\rVert $.
\autoref{thm:pushforward} implies that, for any $ v _h \in V _h ^k $,
we have the estimate
\begin{multline*}
  \bigl\lVert\left( \alpha _1 \cdots \alpha _k \right) ^{ 1/2 } \left(
    \alpha _{ k + 1 } \cdots \alpha _m \right) ^{ -1/2 } \bigr\rVert
  ^{-1} _\infty \left\lVert v _h  \right\rVert _h \\
  \leq \left\lVert i _h v _h \right\rVert \leq \bigl\lVert \left(
    \alpha _1 \cdots \alpha _{ m - k } \right) ^{ 1/2 } \left( \alpha
    _{ m - k + 1 } \cdots \alpha _m \right) ^{ - 1/2 } \bigr\rVert
  _\infty \left\lVert v _h \right\rVert _h ,
\end{multline*}
and since $ J _h = i _h ^\ast i _h $, this implies
\begin{multline*}
  \bigl\lVert \alpha _1 \cdots \alpha _k \left( \alpha
    _{ k + 1 } \cdots \alpha _m \right) ^{ -1 } \bigr\rVert
  ^{-1} _\infty \left\lVert v _h  \right\rVert _h \\
  \leq \left\lVert J _h v _h \right\rVert _h \leq \bigl\lVert 
    \alpha _1 \cdots \alpha _{ m - k } \left( \alpha _{ m - k
      + 1 } \cdots \alpha _m \right) ^{ -1 } \bigr\rVert _\infty
  \left\lVert v _h \right\rVert _h .
\end{multline*}
This bounds the spectrum of the self-adjoint operator $ J _h $, so
finally we obtain a bound on the error term $ \left\lVert I - J _h
\right\rVert $ in terms of the singular values,
\begin{equation}
  \label{eqn:jacobianMax}
\begin{aligned}
  \left\lVert I - J _h \right\rVert \leq \max \Bigl\{ & \left\lvert 1
    - \bigl\lVert \alpha _1 \cdots \alpha _k \left( \alpha _{ k + 1 }
      \cdots \alpha _m \right) ^{ -1 } \bigr\rVert ^{-1} _\infty
  \right\rvert , \\
  & \left\lvert 1 - \bigl\lVert \alpha _1 \cdots \alpha _{ m - k }
    \left( \alpha _{ m - k + 1 } \cdots \alpha _m \right) ^{ -1 }
    \bigr\rVert _\infty \right\rvert \Bigr\} .
\end{aligned} 
\end{equation}
It follows that, if each singular value satisfies $ \left\lvert 1 -
  \alpha _i \right\rvert \leq C h ^{ s + 1 } $, then $ \left\lVert I -
  J _h \right\rVert \leq C h ^{ s + 1} $ as well, and moreover this
will hold for every $ k = 0, \ldots, m $.  Obtaining such bounds on
the singular values, for particular choices of $ \varphi _h $, will be
the topic of the next subsection.

\subsection{Tubular neighborhoods and Euclidean hypersurfaces}

Suppose that $ \left( N , \gamma \right) $ is an oriented,
$n$-dimensional Riemannian manifold, and let $ j \colon M
\hookrightarrow N $ be the inclusion of a submanifold $M$, endowed
with the metric $ g = j ^\ast \gamma $ inherited from $N$.  If $M$ is
compact, then it is possible to construct a \emph{tubular neighborhood}
$U$ around $M$; this is diffeomorphic to an open neighborhood of the
zero section of the normal bundle of $M$, so there is a normal
projection map $ a \colon U \rightarrow M $.  In particular, there
exists some $ \delta _0 > 0 $ such that the set $ M _{ \delta _0 } $,
consisting of points in $N$ whose Riemannian distance to $M$ is less
than $\delta _0 $, is contained in $U$.  (For details, see, e.g.,
\citet{AbMa1978,Lang2002,Lee1997}.)  Now, let $ j _h \colon M _h
\hookrightarrow N $ be a family of inclusions of $m$-dimensional
submanifolds $ M _h $, parametrized by $h$, each endowed with the
Riemannian metric $ g _h = j _h ^\ast \gamma $.  If $ M _h $ lies
inside the tubular neighborhood $U$ and is transverse to $a$ (i.e., $
M _h $ corresponds to a section of $a$), then it is possible to define
the diffeomorphism $ \varphi _h = a \rvert _{ M _h } \colon M _h
\rightarrow M $.

An important case is when $ N = \mathbb{R}^n $, where $ n = m + 1 $
and $ \gamma $ is the standard Euclidean metric, so that $M \subset
\mathbb{R} ^n $ is an oriented Euclidean hypersurface.  It is possible
to define a signed distance function $ \delta \colon U \rightarrow
\mathbb{R} $ on the tubular neighborhood, so that $ \left\lvert \delta
  (x) \right\rvert = \operatorname{dist} \left( x, M \right) $ and $
\nabla \delta (x) = \nu (x) $ is the outward-facing unit normal to $M$
at $ a (x) $.  Every point $ x \in U $ in the tubular neighborhood has
a unique decomposition
\begin{equation*}
x = a (x) + \delta (x) \nu (x) ,
\end{equation*}
so the normal projection map $ a \colon U \rightarrow M $ can be
written as
\begin{equation*}
  a (x) = x - \delta (x) \nu (x) .
\end{equation*}
Therefore,
\begin{equation*}
  \nabla a = I - \nabla \delta \otimes \nu - \delta \nabla  \nu = I -
  \nu \otimes \nu - \delta \nabla \nu = P + \delta S ,
\end{equation*}
where $P = I - \nu \otimes \nu $ is the projection map onto $ T M $
and $ S = - \nabla \nu = - \nabla ^2 \delta $ is the shape operator,
or Weingarten map.  (Note that \citet{Dziuk1988,DeDz2007,Demlow2009}
define a Weingarten map $ H = - S $ using the opposite sign
convention, but this is less common in the differential geometry
literature.)

Instead of directly computing the tangent map $ T a \colon U
\rightarrow M $, we can look at its adjoint, which ``lifts'' vectors
on $M$ to those on $U$.  Given the pullback map $ a ^\ast \colon
\Omega ^1 (M) \rightarrow \Omega ^1 (U) $, the metric can then be used
to identify covectors with vectors, thereby obtaining a pullback map
of vector fields $ \mathfrak{X} (M) \rightarrow \mathfrak{X} (U) $.
Specifically, let $ Y \in T _y M $ and $ x \in a ^{-1} (y) \subset U
$.  Then define the lifted vector $ a ^\ast Y \in T _x U $ satisfying
\begin{equation*}
  X \cdot a ^\ast Y = T a (X) \cdot Y, \quad \forall X \in T _x U .
\end{equation*}
In terms of the Riemannian sharp and flat maps, this can be written as
\begin{equation*}
  \left[ a ^\ast (Y) \right] ^\flat = a ^\ast \bigl( Y ^\flat \bigr)
  \Longleftrightarrow a ^\ast Y = \bigl[ a ^\ast \bigl( Y ^\flat
    \bigr) \bigr] ^\sharp .
\end{equation*}
The following theorem allows us to compute this lifted vector
explicitly, in terms of the signed distance function and shape
operator.

\begin{theorem}
  Let $M$ be an oriented, compact, $m$-dimensional hypersurface of $
  \mathbb{R} ^{ m + 1 } $ with a tubular neighborhood $U$.  If $ Y \in
  T _y M $ and $ x \in a ^{-1} (y) \subset U $, then the lifted vector
  $ a ^\ast Y \in T _x U $ satisfies
  \begin{equation*} 
    a ^\ast Y = \left( I + \delta S \right) Y .
  \end{equation*}
\end{theorem}

\begin{proof}
  Extend $Y$ to a constant vector field on $\mathbb{R}^{m+1}$, so that
  $ Y = \nabla \psi (y) $ for the scalar function $ \psi (x) = x \cdot
  Y $.  Using the definition of the gradient $ \nabla \psi = \left(
    \mathrm{d} \psi \right) ^\sharp $, and the fact that the exterior
  derivative $ \mathrm{d} $ commutes with pullback, we have the
  following chain of equalities:
  \begin{equation*}
    a ^\ast Y = a
    ^\ast \left( \nabla \psi \right) = \left[ a ^\ast \left( \mathrm{d}
        \psi \right) \right] ^\sharp = \left[ \mathrm{d} \left( a ^\ast
        \psi \right) \right] ^\sharp = \nabla \left( \psi \circ a \right)
    .
  \end{equation*} 
  Therefore, applying the chain rule, we get
  \begin{equation*}
    a ^\ast Y = \nabla a (x) \cdot \nabla \psi \left( a (x) \right) =
    \left( P + \delta S \right) Y = \left( I + \delta S \right) Y ,
  \end{equation*}
  where the last equality follows from $ P Y = Y $.
\end{proof}

Finally, when $ x \in M _h $, we can restrict to $ T _x M _h $ by
composing with the adjoint of $ j _h $, i.e., the projection $ P _h =
I - \nu _h \otimes \nu _h $, which gives
\begin{equation*}
  Y _h = j _h ^\ast a ^\ast Y = P _h \left( I + \delta S \right) Y .
\end{equation*}
This map $ j _h ^\ast a ^\ast = P _h \left( I + \delta S \right) $ is
immediately seen to be the adjoint of the restricted tangent map $ T
\varphi _h = T \left( a \rvert _{ M _h } \right) = T \left( a \circ j
  _h \right) = T a \circ T j _h $.  In the next theorem, we bound the
singular values of this map, thereby obtaining an estimate for the
error term $ \left\lVert I - J _h \right\rVert $ in the case of
Euclidean hypersurfaces.

\begin{theorem} \label{thm:jacobianBound}
  Given an oriented, compact, $m$-dimensional hypersurface $ M \subset
  \mathbb{R} ^{ m + 1 } $ with a tubular neighborhood $U$, let $ M _h
  $ be a family of hypersurfaces lying in $U$ and transverse to its
  fibers, such that $ \left\lVert \delta \right\rVert _\infty ,
  \left\lVert \nu - \nu _h \right\rVert _\infty \rightarrow 0 $ as $ h
  \rightarrow 0 $.  Then, for sufficiently small $h$,
  \begin{equation*}
    \left\lVert I - J _h \right\rVert \leq C \left( \left\lVert \delta
      \right\rVert _\infty + \left\lVert \nu - \nu _h \right\rVert
      _\infty ^2 \right) .
  \end{equation*}
\end{theorem}

\begin{proof}
  To obtain bounds on $ Y _h = P _h a ^\ast Y $, and hence on the
  singular values, suppose without loss of generality that $
  \left\lvert Y \right\rvert = 1 $.  By the triangle inequality,
\begin{equation*}
  \left\lvert 1 - \left\lvert Y _h \right\rvert ^2 \right\rvert \leq
  \left\lvert 1 - \left\lvert a ^\ast Y \right\rvert ^2 \right\rvert +
  \left\lvert \left\lvert a ^\ast Y \right\rvert ^2 - \left\lvert Y _h
    \right\rvert ^2 \right\rvert .
\end{equation*}
For the first term, the eigenvalues of the shape operator are the
principal curvatures $ \kappa _1 \left( x \right) , \ldots, \kappa _m
(x) $ for a surface parallel to $M$ at $x$; as noted in
\citet{DeDz2007,Demlow2009}, these are related to the principal
curvatures at $ a (x) \in M $ by
\begin{equation*}
  \kappa _i (x) = \frac{ \kappa _i \left( a (x) \right) }{ 1 + \delta
    (x) \kappa _i \left( a (x) \right) } .
\end{equation*}
It follows that the eigenvalues of $ \delta S $ at $x$ can be
estimated by
\begin{equation*}
  \left\lvert \frac{ \delta (x) \kappa _i \left( a (x) \right) }{ 1 + \delta
      (x) \kappa _i \left( a (x) \right) } \right\rvert = \left\lvert
    1 - \frac{ 1 }{ 1 + \delta (x) \kappa _i \left( a (x) \right) }
  \right\rvert \leq C \left\lvert \delta (x) \right\rvert .
\end{equation*}
Since $ a ^\ast Y = \left( I + \delta S \right) Y $ and $ \left\lvert
  Y \right\rvert = 1 $, this immediately implies
\begin{equation*}
  \left\lvert 1 - \left\lvert a ^\ast Y \right\rvert ^2 \right\rvert
  \leq C \left\lvert \delta \right\rvert .
\end{equation*}
For the remaining term, observe that since $ Y _h = P _h a ^\ast Y $,
\begin{equation*}
  \left\lvert Y _h \right\rvert ^2 = \left\lvert a ^\ast Y - \nu _h
    \left( \nu _h \cdot a ^\ast Y \right) \right\rvert ^2 =
  \left\lvert a ^\ast Y \right\rvert ^2 - \left( \nu _h \cdot a ^\ast
    Y \right) ^2 ,
\end{equation*}
and therefore
\begin{equation*}
  \left\lvert \left\lvert a ^\ast Y \right\rvert ^2 - \left\lvert Y _h
    \right\rvert ^2 \right\rvert = \left( \nu _h \cdot a ^\ast Y
  \right) ^2 = \left( P \nu _h \cdot a ^\ast Y \right) ^2 \leq
  \left\lvert P \nu _h \right\rvert ^2 \left\lvert a ^\ast Y
  \right\rvert ^2 .
\end{equation*}
Now,
\begin{equation*}
  \left\lvert P \nu _h \right\rvert ^2 = \left\lvert \nu _h - \nu
    \left( \nu \cdot \nu _h \right) \right\rvert ^2 = 1 - \left( \nu
    \cdot \nu _h \right) ^2 = \left( 1 + \nu \cdot \nu _h \right)
  \left( 1 - \nu \cdot \nu _h \right) \leq 2 \left( 1 - \nu \cdot \nu
    _h \right),
\end{equation*}
and since $ 2 \left( 1 - \nu \cdot \nu _h \right) = \left\lvert \nu -
  \nu _h \right\rvert ^2 $, it follows that
\begin{equation*}
  \left\lvert \left\lvert a ^\ast Y \right\rvert ^2 - \left\lvert Y _h
    \right\rvert ^2 \right\rvert \leq \left\lvert \nu - \nu _h
  \right\rvert ^2 \left\lvert a ^\ast Y \right\rvert ^2 \leq \left\lvert
    \nu - \nu _h \right\rvert ^2 \left( 1 + \left\lvert 1 -
      \left\lvert a ^\ast Y \right\rvert ^2 \right\rvert \right) \leq C
  \left\lvert \nu - \nu _h \right\rvert ^2 .
\end{equation*}
Putting these together, we have
\begin{equation*}
  \left\lvert 1 - \left\lvert Y _h \right\rvert ^2 \right\rvert \leq C
  \left( \left\lvert \delta \right\rvert + \left\lvert \nu - \nu _h
      \right\rvert ^2 \right) ,
\end{equation*}
from which it follows that at each $ x \in M _h $, the singular values
satisfy
\begin{equation*}
  \left\lvert 1 - \alpha _i \right\rvert \leq C \left( \left\lvert
      \delta \right\rvert + \left\lvert \nu - \nu _h \right\rvert ^2
  \right) , \quad i = 1 , \ldots, m .
\end{equation*}
Finally, applying \eqref{eqn:jacobianMax}, we obtain the uniform bound
\begin{equation*}
  \left\lVert I - J _h \right\rVert \leq C \left( \left\lVert \delta
    \right\rVert _\infty + \left\lVert \nu - \nu _h \right\rVert
    _\infty ^2 \right) ,
\end{equation*}
which completes the proof.
\end{proof}

We now apply this theory to an important class of examples, where $ M
_h $ corresponds to a family of piecewise-linear triangulations (as in
\citet{Dziuk1988,DeDz2007}), or more generally, to the family of
approximate surfaces obtained by degree-$s$ Lagrange interpolation
over each element of the triangulation (as in \citet{Demlow2009}),
where the piecewise-linear case corresponds to $ s = 1 $.  Here, the
elements of this triangulation are assumed to be ``shape-regular and
quasi-uniform of diameter $h$'' \citep{Demlow2009}.  Note that $ M _h
$ is always constructed from an underlying piecewise-linear
triangulation, even in the case of higher-order polynomial
interpolation.  Thus, by \autoref{cor:inducedfeec}, we can define the
$ \mathcal{P} _r ^- $ and $ \mathcal{P} _r $ families of finite
element differential forms on $ M _h $, and obtain bounded cochain
projections, even when $ s > 1 $.

By \citet[Proposition 2.3]{Demlow2009}, for sufficiently small $h$,
the surfaces $ M _h $ obtained by degree-$s$ Lagrange interpolation
satisfy
\begin{equation}
  \label{eqn:demlowBounds}
  \left\lVert \delta \right\rVert _\infty \leq C h ^{ s + 1 } , \qquad
  \left\lVert \nu - \nu _h \right\rVert _\infty \leq C h ^s .
\end{equation}
Therefore, we obtain the following corollary to
\autoref{thm:jacobianBound}.

\begin{corollary}
  \label{cor:surfaceApprox}
  If $ M _h $ is a family of surfaces approximating $M$, obtained by
  degree-$s$ Lagrange interpolation, then $ \left\lVert I - J _h
  \right\rVert \leq C h ^{ s + 1 } $.
\end{corollary}

\begin{proof}
  Applying \eqref{eqn:demlowBounds}, we have 
  \begin{equation*}
    \left\lVert I - J _h \right\rVert \leq C \left( \left\lVert \delta
      \right\rVert _\infty + \left\lVert \nu - \nu _h \right\rVert
      _\infty ^2 \right) \leq C h ^{ s + 1 } + C h ^{ 2 s }
    \leq C h ^{ s + 1 } ,
  \end{equation*} 
  which completes the proof.
\end{proof}

This result generalizes \citet[Proposition 4.1]{Demlow2009}---which
applies only to scalar functions ($ k = 0 $) on hypersurfaces of
dimension $ m = 2, 3 $---to hold for arbitrary $k$-forms, $ k = 0,
\ldots, m $, on hypersurfaces of any dimension.  In particular, the
special case $ k = 0 $, $ m = 2 $, $ s = 1 $, gives $ \left\lVert I -
  J _h \right\rVert \leq C h ^2 $, which recovers the original
estimate of \citet{Dziuk1988} for piecewise-linear triangulation of
surfaces in $\mathbb{R}^3$.  The correspondence between this
framework, and that of Dziuk and Demlow, will be made explicit in the
following worked example.

\begin{example}[The Hodge--Laplace operator on a 2-{D} surface]
Let $M$ be a closed, two-dimensional surface, embedded in
$\mathbb{R}^3$, and suppose the approximate surface $ M _h $ is
obtained from degree-$s$ Lagrange interpolation over a
piecewise-linear triangulation $ \mathcal{T} _h $.  Assume that $
\mathcal{T} _h $ is contained in a tubular neighborhood of $M$, that
its vertices lie on $M$, and that its triangles are shape-regular and
quasi-uniform of diameter $h$.

  Take the continuous Hilbert complex to be the $ L ^2 $-de~Rham
  complex on $M$, i.e., $ W = L ^2 \Omega (M) $ and $ V = H \Omega (M)
  $.  Since $ \mathcal{T} _h $ is piecewise-linear and simplicial, we
  can take the discrete complex to be any of those considered in
  \citet{ArFaWi2006,ArFaWi2010}.  For this example, let us take $ V _h
  ^k $ to be the space of $ \mathcal{P} _r $ $k$-forms, and $ V _h ^{
    k - 1 } $ to be the space of $ \mathcal{P} _{r+1} $ $(k-1)$-forms.
  We emphasize that the fact that $ \mathcal{T} _h $ is a surface
  embedded in $\mathbb{R}^3$, rather than a flat region in
  $\mathbb{R}^2$, does not introduce any additional complications as
  far as the discrete complex is concerned.  Indeed, the shape
  functions are defined with respect to a two-dimensional reference
  triangle, and this reference triangle can be mapped onto a triangle
  embedded in $\mathbb{R}^3$ just as easily as one in $\mathbb{R}^2$.
  These shape functions can, likewise, be lifted up from $ \mathcal{T}
  _h $ to the curved triangles on the interpolated surface $ M _h $.
  For nodal Lagrange finite elements ($k=0$), this observation was
  made by \citet{Dziuk1988} in the piecewise linear case, leading to
  the development of surface finite elements, while \citet{Demlow2009}
  later extended this argument to higher-order Lagrange polynomials.
  (Similar ideas had also been used in the development of
  isoparametric finite elements for Euclidean domains with curved
  boundaries.)

  Now, given some $ f \in L ^2 \Omega ^k (M) $, we obtain a solution $
  \left( \sigma, u, p \right) $ to the variational problem
  \eqref{eqn:mixedProblem} on $M$.  For the discrete variational
  problem \eqref{eqn:discreteProblem}, we can use the tubular
  neighborhood projection to take $ f _h = a \rvert _{ M _h } ^\ast f
  $, thus obtaining a discrete solution $ \left( \sigma _h , u _h , p
    _h \right) $ on $ M _h $.  The modified discrete solution $ \left(
    \sigma _h ^\prime , u _h ^\prime , p _h ^\prime \right)$---which
  is used only for the analysis, but is not necessary for
  computation---also lives on $ M _h $, while its image $ \left( i _h
    \sigma _h ^\prime , i _h u _h ^\prime , i _h p _h ^\prime \right)
  $ lives on $M$ itself.

  Therefore, assuming sufficient elliptic regularity, the ``improved
  estimates'' of \citet{ArFaWi2006,ArFaWi2010} yield the $ L ^2 $
  estimates for the modified problem,
  \begin{align*}
    \left\lVert u - i _h u _h ^\prime \right\rVert + \left\lVert p - i
      _h p _h ^\prime \right\rVert &\leq C h ^{r+1} \left\lVert f
    \right\rVert _{ H ^{ r - 1 } } , \\
    \left\lVert \mathrm{d} \left( u - i _h u _h ^\prime \right)
    \right\rVert + \left\lVert \sigma - i _h \sigma _h ^\prime
    \right\rVert &\leq C h ^r \left\lVert f \right\rVert _{ H ^{ r -
        1 } }, \\
    \left\lVert \mathrm{d} \left( \sigma - i _h \sigma _h ^\prime
      \right) \right\rVert &\leq C h ^{ r - 1 } \left\lVert f
    \right\rVert _{ H ^{ r -1 } } ,
  \end{align*}
  which can be combined into the single estimate
  \begin{multline*}
    \left\lVert u - i _h u _h ^\prime \right\rVert + \left\lVert p - i
      _h p _h ^\prime \right\rVert + h \left( \left\lVert \mathrm{d}
        \left( u - i _h u _h ^\prime \right) \right\rVert +
      \left\lVert \sigma - i _h \sigma _h ^\prime \right\rVert
    \right) \\
    + h ^2 \left\lVert \mathrm{d} \left( \sigma - i _h \sigma _h
        ^\prime \right) \right\rVert \leq C h ^{ r + 1 } \left\lVert f
    \right\rVert _{ H ^{ r -1 } } .
  \end{multline*} 
  Applying \autoref{cor:surfaceApprox} to account for the surface
  approximation error, we obtain the final error estimate for the
  discrete problem,
  \begin{multline*}
    \left\lVert u - i _h u _h \right\rVert + \left\lVert p - i _h p _h
    \right\rVert + h \left( \left\lVert \mathrm{d} \left( u - i _h u
          _h \right) \right\rVert + \left\lVert \sigma - i _h \sigma
        _h \right\rVert
    \right) \\
    + h ^2 \left\lVert \mathrm{d} \left( \sigma - i _h \sigma _h
      \right) \right\rVert \leq C \left( h ^{ r + 1 } \left\lVert f
      \right\rVert _{ H ^{ r -1 } } + h ^{ s + 1 } \left\lVert f
      \right\rVert \right) .
  \end{multline*} 
  In particular, this implies that choosing \emph{isoparametric
    elements}, with $ r = s $, yields the optimal rate of convergence.

  The case $ k = 0 $ and $ r = s = 1 $ corresponds to the lowest-order
  approximation of the Laplace--Beltrami equation, where $ M _h $ is
  piecewise-linear and $ V _h ^0 $ consists of piecewise-linear ``hat
  functions'' on $M _h $.  In this case, the estimate above becomes
  \begin{equation*}
    \left\lVert u - i _h u _h \right\rVert + \left\lVert p - i _h p _h
    \right\rVert + h \left\lVert \nabla \left( u - i _h u _h
      \right) \right\rVert \leq C h ^2 \left\lVert f \right\rVert ,
  \end{equation*} 
  which precisely recovers the estimate of \citet{Dziuk1988,DeDz2007}.
  More generally, taking $ k = 0 $ with arbitrary $ r $ and $s$, we
  have
  \begin{equation*}
    \left\lVert u - i _h u _h \right\rVert + \left\lVert p - i _h p _h
    \right\rVert + h \left\lVert \nabla \left( u - i _h u _h
      \right) \right\rVert \leq C \left( h ^{ r + 1 } \left\lVert f
      \right\rVert _{ H ^{ r -1 } } + h ^{ s + 1 } \left\lVert f
      \right\rVert \right) ,
  \end{equation*} 
  which agrees with \citet{Demlow2009}.

  On the other hand, we can also extend these estimates to the cases $
  k = 1 $, which corresponds to the mixed formulation of the vector
  Laplacian, and $ k = 2 $, which corresponds to the mixed formulation
  of the scalar Laplacian.  For $ k = 1 $, the estimate for general $
  r $ and $s$ becomes
  \begin{multline*}
    \left\lVert u - i _h u _h \right\rVert + \left\lVert p - i _h p _h
    \right\rVert + h \left( \left\lVert \nabla \times \left( u - i _h
          u _h \right) \right\rVert + \left\lVert \sigma - i _h \sigma
        _h \right\rVert
    \right) \\
    + h ^2 \left\lVert \nabla \left( \sigma - i _h \sigma _h \right)
    \right\rVert \leq C \left( h ^{ r + 1 } \left\lVert f \right\rVert
      _{ H ^{ r -1 } } + h ^{ s + 1 } \left\lVert f \right\rVert
    \right) ,
  \end{multline*}
  while for $ k = 2 $, we obtain
  \begin{multline*}
    \left\lVert u - i _h u _h \right\rVert + \left\lVert p - i _h p _h
    \right\rVert + h \left\lVert \sigma - i _h \sigma _h \right\rVert
    \\
    + h ^2 \left\lVert \nabla \cdot \left( \sigma - i _h \sigma _h
      \right) \right\rVert \leq C \left( h ^{ r + 1 } \left\lVert f
      \right\rVert _{ H ^{ r -1 } } + h ^{ s + 1 } \left\lVert f
      \right\rVert \right) .
  \end{multline*}
\end{example}

\subsection{Other variational crimes}

% [ Mass lumping, boundary conditions, \ldots ]
The variational crimes framework developed in \autoref{sec:crimes} is
quite general, representing a natural extension of the Strang lemmas
from Hilbert spaces to Hilbert complexes.  As such, the standard
``crimes'' that are typically analyzed in Hilbert spaces with the
Strang lemmas may now be analyzed in the setting of Hilbert complexes.
These crimes---including numerical quadrature, approximate
coefficients, approximate boundary data, approximate domains, as well
as isoparametric and other geometric approximations to more complex
domain shapes---can all be represented as an approximate bilinear form
$B_h$, an approximate linear functional $F_h$, and an approximation
subspace $V_h \not\subset V$, as in~\eqref{eqn:generalizedGalerkin}.
In addition, techniques such as \emph{mass-lumping}, which yield a
number of benefits, such as discrete maximum principles and more
efficient evolution algorithms for parabolic equations, are often
analyzed in a similar way, and as such may now be analyzed in Hilbert
complexes through the framework developed in \autoref{sec:crimes}.

\section{Conclusion}

We began the article in \autoref{sec:hilbert} with a review of the
mathematical concepts that play a fundamental role in finite element
exterior calculus, as developed in \citet{ArFaWi2010}; these included
abstract Hilbert complexes and their morphisms, domain complexes,
Hodge decomposition, the Poincar\'{e} inequality, the Hodge Laplacian,
mixed variational problems, and approximation using Hilbert
subcomplexes.  In \autoref{sec:crimes}, we then considered
approximation of a Hilbert complex by a second complex, related to the
first complex through an injective morphism rather than through
subcomplex inclusion.  We developed several key results for this pair
of complexes and the maps between them, and then derived error
estimates for generalized Galerkin-type approximations of solutions to
variational problems using the approximating complex; these estimates
can be viewed as generalizing the results of \citet{ArFaWi2010} to
``external'' approximations.  Our main abstract results are thus
essentially \emph{Strang-type lemmas} for approximating variational
problems in Hilbert complexes.

As an application of the new framework developed in
\autoref{sec:crimes}, we developed a second distinct set of results in
\autoref{sec:diffForms} for the case of the Hodge--de~Rham complex of
differential forms on a compact, oriented Riemannian manifold.  We
first reviewed Hodge--de~Rham theory, and then considered a pair of
Riemannian manifolds related by diffeomorphisms.  We then established
estimates for the maps needed to apply the generalized Hilbert complex
approximation framework from \autoref{sec:crimes}, subsequently
specializing this analysis to the case of a Euclidean hypersurface,
with approximating hypersurfaces living in a tubular neighborhood.
The surface finite element methods, as analyzed in
\citet{Dziuk1988,DeDz2007,Demlow2009}, fit precisely into this class
of approximation problems; as such, we illustrated how our results
recover the analysis framework and \emph{a priori} estimates of
\citet{Dziuk1988,DeDz2007,Demlow2009}, and also extend their results
from scalar functions on 2- and 3-surfaces to general $k$-forms on
arbitrary dimensional hypersurfaces.  Our results also generalize
those earlier estimates from nodal finite element methods for the
Laplace--Beltrami operator to mixed finite element methods for the
Hodge Laplacian.  By analyzing surface finite element methods using a
combination of general tools from differential geometry and functional
analysis, we are led to a more geometric analysis of surface finite
element methods, whereby the main results become more transparent.

There remain a number of interesting and challenging problems that
were not addressed in the current article.  One such problem is the
extension of the pointwise error estimates of \citet{Demlow2009} for
$0$-forms to general $k$-forms; this analysis relies on known results
for the Green's function of the Laplace--Beltrami operator on the
continuous surface (cf.~\citep{Aubin1982}), and analogous results
would be needed for general $k$-forms.  A second problem of interest
is an extension of the Hilbert complex framework to more general
Banach complexes, as would be needed to handle some nonlinear
problems.  This leads to a third interesting problem, which would
involve the extension of the weak-$*$ convergence and contraction
frameworks, used for adaptive finite element methods for
linear~\citep{CaKrNoSi2008,MoSiVe2008} and
nonlinear~\citep{HoTsZh2010} problems, to the setting of finite
element exterior calculus, as well as to the surface finite element
setting.

\section*{Acknowledgments}

  MH was supported in part by NSF DMS/CM Awards 0715146 and 0915220,
  NSF MRI Award 0821816, NSF PHY/PFC Award 0822283, and by DOD/DTRA
  Award HDTRA-09-1-0036.

  AS was supported in part by NSF DMS/CM Award 0715146 and by NSF
  PHY/PFC Award 0822283, as well as by NIH, HHMI, CTBP, and NBCR.

% \bibliographystyle{sternurl}
% \bibliography{HoSt2010}

\end{document}